\documentclass[a4paper,reqno,twoside,11pt]{amsart}

\usepackage{amssymb}
\usepackage{hyperref}
\usepackage{latexsym}
\usepackage{color}
\usepackage{dsfont}

\usepackage{geometry}
\newtheorem{propn}{Proposition}[section]
\newtheorem{thm}[propn]{Theorem}
\newtheorem{lemma}[propn]{Lemma}
\newtheorem{cor}[propn]{Corollary}

\theoremstyle{definition}
\newtheorem{defn}[propn]{Definition}

\newtheorem*{notn}{Notation}

\newtheorem{example}[propn]{Example}

\theoremstyle{remark}
\newtheorem*{rem}{Remark}
\newtheorem*{rems}{Remarks}

 \newcommand{\lft}{{\ell}}

 \newcommand{\hscale}{\mathit{s}_h}
 \newcommand{\seriesprod}{\lhd}
 
\newcommand{\ToyI}{I^\Upsilon}
\newcommand{\FockI}{I^{\Fock}}

\newcommand{\Toyj}{\text{\j}^{\! \Upsilon}}

\newcommand{\Fockj}{\text{\j}^\Fock}
\newcommand{\Toysigma}{\sigma^\Upsilon}
\newcommand{\Focksigma}{\sigma^\Fock}
\newcommand{\Toyrho}{\rho^\Upsilon}
\newcommand{\Fockrho}{\rho^\Fock}

\newcommand{\hXG}{X^{h,G}}
\newcommand{\hXGh}{X^{h,G_h}}
\newcommand{\hXPhn}{X^{h,P(h,n)^n}}

\newcommand{\Xsofth}{X^{(h)}}
\newcommand{\Qsofth}{Q^{(h)}}

 \newcommand{\Fockk}{\mathcal{F}}

\newcommand{\vetrunc}{\wt{\ve}}

\newcommand{\Toyk}{\Upsilon}

\newcommand{\Diil}{\mathsf{Di}}
\newcommand{\Scil}{\mathsf{Sc}}
\newcommand{\Sil}{\mathsf{S}}

\newcommand{\Pil}{\mathsf{P}}
\newcommand{\Il}{\mathsf{I}}

\newcommand{\QIP}{\Delta} %%%%%%%%%%%%%% NEW

\newcommand{\flip}{\mathbin{\widetilde{\otimes}}}

\newcommand{\ve}{\varepsilon}

\newcommand{\Hil}{\mathsf{H}}
\newcommand{\hil}{\mathsf{h}}
\newcommand{\Kil}{\mathsf{K}}
\newcommand{\kil}{\mathsf{k}}

\newcommand{\Til}{\mathsf{T}}
\newcommand{\Al}{\mathsf{A}}

\newcommand{\init}{\mathfrak{h}}
 \newcommand{\noise}{\mathsf{k}}
 
 \newcommand{\khat}{{\wh{\noise}}}
  \newcommand{\Khat}{{\wh{\Kil}}}
 \newcommand{\ktildehat}{\wh{\ktilde}}

 \newcommand{\Fock}{\mathcal{F}}
 \newcommand{\Exps}{\mathcal{E}}

 \newcommand{\Step}{\mathbb{S}}

\newcommand{\chat}{\wh{c}}
\newcommand{\dhat}{\wh{d}}
\newcommand{\ghat}{\wh{g}}
\newcommand{\fhat}{\wh{f}}

\newcommand{\ktilde}{\wt{\kil}}

\newcommand{\Deltatilde}{\wt{\Delta}}

\newcommand{\Ftilde}{\wt{F}}
\newcommand{\Gtilde}{\wt{G}}

\newcommand{\rd}{\mathrm{d}}

\newcommand{\Real}{\mathbb{R}}
\newcommand{\Rplus}{\Real_+}
\newcommand{\Comp}{\mathbb{C}}
\newcommand{\Nat}{\mathbb{N}}

\newcommand{\Integers}{\mathbb{Z}}
\newcommand{\Zplus}{\Integers_+}

\newcommand{\ip}[3][]{#1\langle #2, #3 #1\rangle}
\newcommand{\norm}[1]{\lVert #1 \rVert}
\newcommand{\cbnorm}[1]{\lVert #1 \rVert_\cb}

\newcommand{\ket}[1]{\vert #1 \rangle}

\newcommand{\indset}{1}

\newcommand{\cb}{{\text{\tu{cb}}}}

\newcommand{\sa}{{\text{\tu{sa}}}}

\newcommand{\wh}{\widehat}
\newcommand{\wt}{\widetilde}

\newcommand{\ot}{\otimes}
\newcommand{\otul}{\mathbin{\underline{\ot}}\,}
\newcommand{\otol}{\mathbin{\overline{\ot}}\,}

\newcommand{\op}{\oplus}

\newcommand{\les}{\leqslant}
\newcommand{\ges}{\geqslant}

\newcommand{\tu}{\textup}

\DeclareMathOperator{\reversed}{r}

\DeclareMathOperator{\Ran}{Ran}
\DeclareMathOperator{\Lin}{Lin}

\DeclareMathOperator{\Conv}{Conv}

\DeclareMathOperator{\id}{id}
\DeclareMathOperator{\re}{Re}
\DeclareMathOperator{\im}{Im}

 \DeclareMathOperator{\mesh}{mesh}

 \DeclareMathOperator{\loc}{loc}
\DeclareMathOperator{\Spec}{Spec}

 \newcommand{\dualco}{X^{\sharp}}
 \newcommand{\reversedco}{X^{\reversed}}
  \newcommand{\dualqrw}{W^{\sharp}}

\newenvironment{alist}
{

\begin{enumerate}}
{\end{enumerate}}

\newenvironment{rlist}
{

\begin{enumerate}}
{\end{enumerate}}

\numberwithin{equation}{section}
\begin{document}

\title[Convergence of quantum random walks]
{Strong convergence of quantum random walks
\\
via semigroup decomposition}
\author[Belton, Gnacik and Lindsay]
{Alexander C.R.\ Belton, Micha\l\ Gnacik and J.\ Martin Lindsay}
\address{Department of Mathematics \& Statistics
\\
Lancaster University
\\  Lancaster LA1 4YF \\ UK}
\email{a.belton@lancaster.ac.uk}
\email{j.m.lindsay@lancaster.ac.uk}

\address{\noindent Department of Mathematics
	\\
	Lion Gate Building, Lion Terrace
	\\
	The University of Portsmouth
	\\\noindent  Portsmouth PO1 3HF \\ UK}
\email{michal.gnacik@port.ac.uk}


%%%%%%%%%%%%%%%%%%%%%%%%%%%%%%%%%%%      DATE
%\date{
%{\color{red}
%07xii2017 + (for editing purposes)}}

%%%%%%%%%%%%%%%%%%%%%%%%%%%%%%%%%%%      ABSTRACT
\begin{abstract}
 We give a simple and direct treatment of
the strong convergence of quantum random walks to quantum stochastic operator cocycles,
via the semigroup decomposition of such cocycles.
Our approach also delivers convergence of the pointwise product of quantum random walks
 to the quantum stochastic Trotter product of the respective limit cocycles,
 thereby revealing the algebraic structure of the limiting procedure.
The repeated quantum interactions model  is shown to fit nicely into the convergence scheme described.
\end{abstract}

\maketitle

%\vspace*{-9cm}
%\begin{flushright}
%\ti{\small To appear in:
%\\
% Annales Henri Poincar\'e}
%\end{flushright}
%\vspace{8cm}

\tableofcontents
%%%%%%%%%%%%%%%%%%%%%%%%%%%%%%%%%%%      TABLE OF CONTENTS
% \tableofcontents

%%%%%%%%%%%%%%%%%%%%%%%%%%%%%%%%%%%%%%%%%%%% SECTION
\section*{Introduction}
\label{section: intro}

Quantum random walks have been a feature of noncommutative probability
for over twenty-five years; as emphasised in~\cite{BvH08},
 ``the convergence of discrete quantum Markov chains to
continuous ones is a fundamental problem in quantum probability''.
In
Meyer's book (\cite{Meyer}), Journ\'e is credited as the first to use
discrete approximations to the relevant symmetric Fock space and to quantum stochastic
processes.
Around the same time
a central-limit theorem, yielding the quantum harmonic oscillator
as a limit of quantum Bernoulli processes,
was proved (\cite{AcB89}; see~\cite{Mey89}).
One should also mention
von Waldenfels' earlier use of
discrete approximation to define
quantum L\'evy processes on
unitary matrix groups as multiplicative It\^o integrals (\cite{vWa84}).
In further early work, it was shown that certain quantum stochastic
flows, which are generalisations of classical diffusions, may be
approximated by so-called spin random walks
(\cite{LiP}); see also~\cite{Par}, and~\cite{Sinha}.
More recently, a theory of quantum random walks generated by completely bounded maps on
operator spaces was developed,
in an approach which admits the treatment of particle algebras in an
arbitrary normal state ([$\text{B}_{1\!-\!3}$]).
The theory was then extended to
quantum random walks in Banach algebras,
further elucidating the way in which the limits arise (\cite{DaL}).
The approach to discrete approximation in~\cite{BvH08}
is in the spirit of the current paper,
and may be viewed as an unbounded-generator counterpart
in which
the Trotter-Kato theorem is exploited in place of Euler's exponential formula.

These
convergence theorems are analogues of Donsker's
invariance principle, with the limit process being a quantum stochastic
cocycle,
\emph{i.e.}~the quantum stochastic analogue of
a stochastic semigroup in the sense of Skorohod (\cite{Skorohod}),
rather than a classical Wiener process.
As well as their probabilistic interpretation as noncommutative Markov
chains, quantum random walks may also be seen as models for the
dynamics of a quantum-mechanical system undergoing repeated
interactions with an environment composed of an infinite number of
identical particles.
This point of view was adopted in~\cite{AtP} and~\cite{AJrepeated};
links between the
repeated-interactions model and time-ordered exponentials (\cite{Holevo1})
were demonstrated in~\cite{Gough}.
Our approach is readily modified to the convergence of mapping-valued
(as opposed to operator-valued) quantum random walks,
and thereby to
the discrete approximation of quantum L\'evy processes.

There have been many applications of quantum random walks: to quantum
filtering and quantum feedback control (\cite{GoS04}, \cite{BvHJ09});
to the approximation of L\'evy processes on quantum groups (\cite{FrS},
\cite{LiS}); to the construction of dilations of quantum dynamical
semigroups (\cite{Sah}, \cite{BQRW}).
Repeated-interactions models for the one-atom maser,
an important system in quantum optics (\cite{Gardiner}),
have been investigated in~\cite{BJM06} and~\cite{BrP09};
in contrast
to the results we prove below, the convergence theorems obtained in these papers
give only the reduced dynamics of the limit system and disregard the
limit behaviour of the environment.
Interesting connections between noncommutative Markov chains and
multivariate operator theory were explored in~\cite{Goh10}.
We should alert the reader to the fact that
there
 are several other notions of quantum random walk
in the literature,
for example
`quantum walk on a graph',  `unitary random walk',
in particular `Hadamard walk'
(\cite{AAKV},~\cite{Konno},~\cite{Kem}),
and
`open quantum random walk' (\cite{At+}).
The approximation of continuous-time quantum random walks by discrete-time walks
is addressed
in~\cite{Childs}, for the former type,
and in~\cite{Pell}, for the latter.

For us here,
a quantum random walk is a
discrete-time, bi-adapted covariant quantum stochastic evolution,
or
discrete-time quantum stochastic cocycle
(Definition~\ref{defn: QRW}).
Adaptedness and covariance of the quantum random walk are
with respect to
the natural operator filtration of, and the time shift on,
the algebra of bounded operators on
a toy Fock space (introduced in Section~\ref{section: qrws}).
The limiting objects are
(continuous-time) bi-adapted covariant quantum stochastic evolutions,
or quantum stochastic cocycles (Definition~\ref{defn: cocycle}).
Adaptedness and covariance of the QS cocycle are
with respect to
the natural operator filtration of, and the time shift on,
the algebra of bounded operators on
a symmetric Fock space
with test functions from
an $L^2$-space of Hilbert space-valued functions on the half-line.
Thus the notion of independence implicit here is that of tensor independence,
as opposed to free independence, or freeness (\cite{VDN}), for example.

A central feature of this work is
the exploitation of what has come to be known as the `semigroup approach'
([$\text{LW}_{2,3}$]; see~\cite{Lbook}).
Specifically,
we use the semigroup decomposition of continuous-time quantum stochastic cocycles
(given in~\eqref{eqn: decomp}) and Euler's exponential formula \eqref{eqn: euler}
to give a new, direct, and considerably simplified proof of the
convergence of suitably scaled quantum random walks to quantum
stochastic cocycles.
Properties of a certain nonlinear transformation on block matrix operators
which we refer to as the Holevo transform
(Theorem~\ref{thm: Ghn} and Proposition~\ref{propn: Q skew}),
and a key observation on compositions (Theorem~\ref{propn: composition}),
accompany our main convergence theorem (Theorem~\ref{thm: main}).
Together these
lead to the realisation of a general class of quantum stochastic cocycles as
scaled limits of quantum random walks of the corresponding kind,
that is, contractive, isometric, or unitary (Theorem~\ref{thm: strong cgce}).
They also
yield short and transparent demonstrations of strengthened forms of results
on the repeated-interactions model (\cite{AtP},~\cite{ADP}).
Specifically,
in Theorem~\ref{thm: rqi} we generalise Theorem 19 of~\cite{AtP},
dispensing with underlying Hilbert--Schmidt-type assumptions on the components of
the generator of the limiting stochastic cocycle,
and in Theorem~\ref{thm: bipart} we generalise Theorem 3.1 of~\cite{ADP}
by avoiding any restriction on the dimension of the noise
whilst
allowing scattering in the interaction Hamiltonians.
Our results are coordinate-free throughout.

\medskip
\emph{Outline}. The structure of the paper is as follows.
Following a
background section on quantum stochastic operator cocycles,
Section~\ref{section: qrws} describes the very close analogy between such cocycles
and quantum random walks on a Hilbert space.
After a short section on the scaled embedding of QRWs
as continuous-time processes on a Fock space,
Section~\ref{section: standard qrw approx} contains the new proof of our central result,
and its corollary on the approximation of quantum stochastic flows
by QRWs on the algebra of bounded operators on a Hilbert space,
\emph{i.e.}~the Heisenberg picture.
The algebraic structure of the approximation scheme is exposed
in Section~\ref{section: composition of quantum random walks}.
In Section~\ref{section: transformation},
we discuss
a basic (nonlinear) transformation on block matrix operators
which we refer to as \emph{the Holevo transform};
it provides means for some of the realisations of the approximation scheme
given in Section~\ref{section: realisation}.
In Section~\ref{section: repeated quantum interactions} we show how
the repeated quantum interactions model, and entanglement of bipartite systems,
fit nicely into the general scheme developed here.

In a sister paper (\cite{BGLZ}),
we consider embeddings of toy Fock space appropriate to
faithful states on a particle algebra, and
obtain quasifree stochastic cocycles, in the sense of~\cite{LM2}, as
limits of scaled random walks in that setting.

%%%%%%%%%%%%%%%%%%%%%%%%%%%%%%%%%%%%%%%%%%%% SECTION
%\section{Notation}
%\label{section: notation}

\medskip
\emph{Notation}.
For a vector-valued function $g: S \to V$ and subset $A$ of $S$,
$g_A$ denotes the function $S \to V$ which
agrees with $g$ on $A$ and vanishes elsewhere,
extending the standard notation $\indset_A$ for the indicator function of $A$.
We make extensive use of
the following extension to
(the mathematician's version of)
the Dirac bra-ket notation.
For a vector $u$ in a Hilbert space $\hil$,
the operator $I_\Hil \ot \ket{u}: \Hil \to \Hil \ot \hil$ given by
$\xi \mapsto \xi \ot u$,
is denoted $E_u$; its adjoint is denoted $E^u$.
The Hilbert space $\Hil$ is always
clear from the context.
We denote
the space of bounded operators from $\Hil$ to a Hilbert space $\Kil$ by $B(\Hil; \Kil)$,
abbreviating $B(\Hil;\Hil)$ to $B(\Hil)$,
and write
$B(\Hil)_{\sa}$ for the space of selfadjoint operators on $\Hil$,
and $\re T$, respectively $\im T$, for
the real part $\tfrac{1}{2}(T+T^*)$ and imaginary part $\tfrac{1}{2i}(T-T^*)$ of an operator $T \in B(\Hil)$.
The algebraic and ultraweak tensor products are denoted
$\otul$ and $\otol$ respectively
and, for vectors $\zeta$ and $\eta$ in a Hilbert space $\hil$,
the vector functional
$T \mapsto \ip{\zeta}{T \eta}$ on $B(\hil)$ is denoted $\omega_{\zeta,\eta}$,
or $\omega_\zeta$ if $\eta = \zeta$.
As usual, $B(\hil)_*$ denotes the space of
ultraweakly continuous functionals on $B(\hil)$.
We write
$\Ran$, $\Spec$ and $\Conv$ respectively for range, spectrum and convex hull.
%The Riesz--Nagy symbol $\smile$ is used to denote `commutes with',
%for bounded operators on a Hilbert space (\cite{RiN}).
For
the symmetric Fock space over a Hilbert space,
exponential vectors,
and second quantisation
we use the following notations.
Let
$\hil^{\vee n}$
denote the $n$-fold symmetric tensor power of a Hilbert space $\hil$,
with the convention $\hil^{\vee 0} :=  \Comp$,
then,
for $u\in \hil$,
Hilbert spaces $\hil_1$ and $\hil_2$
and
$C \in B(\hil_1; \hil_2)$,
\begin{equation}
\label{eqn: 2ndQ}
\Gamma(\hil) :=
\bigoplus_{n\geq 0} \hil^{\vee n},
\  \
 \ve(u):= \big( (n!)^{-1/2}u^{\ot n}\big)_{n\ges 0}
\ \text{ and } \
\Gamma(C) := \bigoplus_{n \geq 0} C^{\ot n},
\end{equation}
where the latter is
viewed as an operator from $\Gamma(\hil_1)$ to $\Gamma (\hil_2)$.
Since
$\norm{ C^{\ot n} } = \norm{C}^n$ for all $n \in \Zplus$,
$\Gamma(C)$ is a contraction if $C$ is,
and is unbounded otherwise.
Second quantisation
enjoys the following functorial properties: for compatible contraction operators $C_1$ and $C_2$,
\[
\Gamma( I_\hil ) = I_{\Gamma(\hil)}, \
\Gamma( C^* ) =  \Gamma( C)^*, \
\Gamma( C_1 C_2 ) = \Gamma( C_1 ) \Gamma( C_2)
\ \text{ and } \
\Gamma (C) \ve(u) = \ve( Cu ).
\]

\emph{Fix Hilbert spaces
$\init$ and $\noise$},
referred to
as the `initial space' or `system space', and the `noise dimension space' respectively.
The following notations are used throughout:
\begin{align*}
&
\chat := \binom{1}{c} \in \khat := \Comp \oplus \noise
\quad
(c \in \noise),
\ \text{ thus } \
\wh{0} :=  \binom{1}{0},
\ \text{ and }
\\
&
\QIP :=
I_\init \ot ( 0_\Comp \op I_{\noise} ) =
\begin{bmatrix} 0 & 0 \\ 0 & I_{\init \ot \noise} \end{bmatrix},
\ \text{ thus } \
 \Delta^{\perp} =
 \begin{bmatrix} I_{\init} & 0 \\ 0 & 0 \end{bmatrix}.
\end{align*}
The Hilbert spaces
$\init \ot \khat$ and $\init \oplus ( \init \ot \noise )$
are identified,
so that
each operator $Q \in B(\init \ot \khat )$ has a block matrix form
$\left[ \begin{smallmatrix} A & C \\ B & D\end{smallmatrix} \right]$.

%%%%%%%%%%%%%%%%%%%%%%%%%%%%%%%%%%%

%%%%%%%%%%%%%%%%%%%%%%%%%%%%%%%%%%%%%%%%%%%% SECTION

\section{Quantum stochastic cocycles}
\label{section: quantum stochastic cocycles}

In this section we briefly recall the basic facts that are
needed concerning quantum stochastic (QS) analysis, and specifically
operator cocycles and their generation via QS differential equations.
We emphasise that by \emph{quantum stochastic process} we mean here
time-indexed family of operators adapted to the natural filtration of
subalgebras of the algebra of bounded operators on a
symmetric Fock space over an $L^2$-space of vector-valued functions,
as in Definition~\ref{defn: cocycle} below.
For more detail, see~\cite{Lbook} which is our basic reference,
and~\cite{LQST2}
where an exposition of the relevant quantum It\^o algebra may be found.
For further background, see~\cite{Partha},~\cite{Meyer} and~\cite{FagnolaNotes}.

For any subinterval $I$ of $\Rplus$, set
\[
\Fock_I = \Fock^\noise_I :=
\Gamma ( L^2(I; \noise)),
\]
abbreviating to $\Fockk = \Fock^\noise$ when $I = \Rplus$.
For any subset $\Til$ of $\noise$,
let $\Step_\Til$ denote the subset of
$L^2(\Rplus; \noise)$
consisting of $\Til$-valued step functions, whose right-continuous versions we always take,
and
set $\Exps_\Til := \Lin \{ \ve(f): f \in \Step_\Til \}$.
(When $\Til = \noise$ we abbreviate to $\Step$ and $\Exps$.)
The subspace $\Exps_\Til$ is dense in $\Fockk$ if and only if the set $\Til$ is total and contains $0$
(\cite{Skeide}; see~\cite{Lbook}, Proposition 2.1).
A typical example of $\Til$ is an orthonormal basis augmented by the vector $0$.
The natural identification
\begin{equation}
\label{eqn: tensor decomp}
\Fockk = \Fockk_{[0,r[} \ot  \Fockk_{[r,t[} \ot  \Fockk_{[t,\infty[}
\qquad
(r,t \in \Rplus, r \les t)
\end{equation}
witnessed by exponential vectors,
$\ve(f) = \ve(f|_{[0,r[}) \ot  \ve(f|_{[r,t[}) \ot  \ve(f|_{[t,\infty[})$,
is frequently invoked.
We use the notation
$\FockI_{[r,t[}$
for the identity operator on $\Fockk_{[r,t[}$.

Two families of endomorphisms of $B(\Fockk)$ are defined by
\[
 \Focksigma_{t} (T) :=
\FockI_{[0,t[} \ot S_t T S_t^*
\ \text{ and } \
 \Fockrho_{t} (T) :=
 R_t T R_t
\quad
(t \in \Rplus)
\]
where
$S_t$
is the shift operator $\Gamma( s_t ) : \Fock \to \Fock_{[t, \infty[}$
 and $R_t$
is the time-reversal operator
$\Gamma( r_t ) : \Fock \to \Fock$,
for the
unitary operator
$s_t: L^2( \Rplus; \noise) \to L^2( [t, \infty[; \noise )$ and
selfadjoint unitary operator
$r_t: L^2( \Rplus; \noise) \to L^2( \Rplus; \noise )$
defined by
\[
(s_t f)(s) = f(s-t)  \text{ for } s \in [t,\infty[
\ \text{ and } \
( r_t f )(s) =
\left\{
  \begin{array}{ll}
f(t-s) & \text{ if } 0 \les s \les t,
\\
f(s) & \text{ if } s > t.
\end{array}
\right.
\]

\begin{defn}
\label{defn: cocycle}
A QS bounded-operator (left) cocycle on $\init$ with noise dimension space $\noise$
is a family of operators $X = (X_t)_{t\ges 0}$ in $B(\init \ot \Fockk)$
satisfying the following adaptedness and cocycle conditions:
\begin{equation*}
X_0 = I_{\init \ot \Fockk},
\ \
X_{r+t} = X_r \sigma_r ( X_t )
\ \text{ and } \
X_t \in B(\init \ot \Fockk_{[0,t[}) \ot \FockI_{[t,\infty[}
\qquad
(r,t \in \Rplus),
\end{equation*}
where $\sigma_r:=  \id_{B(\init)} \otol \Focksigma_r$.
A QS cocycle $X$ is called
\emph{elementary}, or
\emph{Markov regular},
if
\[
s \mapsto X^{f,g}_s
\text{ is continuous }
\qquad
(f,g \in L^2_{\loc}(\Rplus; \noise)).
\]
The notation here is as follows. For a QS process $X$,
\begin{equation}
\label{eqn: pre Xfg}
X^{f,g}_s :=
E^{\ve(f_{[0,s[})} X_s E_{\ve(g_{[0,s[})}
\ \text{ and } \
f_{[0,s[} := \indset_{[0,s[} f.
\end{equation}
A QS cocycle $X$ is called
\emph{contractive,
isometric,
or
unitary}
if each operator $X_t$ has that property;
it is called
\emph{quasicontractive}
if, for some $\beta \in \Rplus$, the QS cocycle
$(e^{-\beta t} X_t )_{t \ges 0}$ is contractive;
in this case
\[
\beta_0(X) :=
\inf \big\{
\beta \in \Real : \ \big\| e^{- \beta t } X_t \big\| \les 1
\ \text{ for all } t \in \Rplus
\big\}
\]
is referred to as the
\emph{exponential growth bound} of
 $X$.
\end{defn}
If $X$ is a QS cocycle then, for each $c$, $d \in \noise$,
\begin{equation}
\label{eqn: assoc semi}
P^{c,d} := ( X^{c_{[0,t[}, d_{[0,t[}}_t )_{t \ges 0}
\end{equation}
defines a semigroup on $\init$.
Here a vector $c$ in $\noise$ is viewed as
 an element of $L_{\loc}^2(\Rplus; \noise)$, with
$c_{[0,t[}$ denoting
the function equal to $c$ on the interval $[0,t[$ and zero outside,
for each $t \in \Rplus$.
If $X$ is quasicontractive then
 $X$ is elementary if and only if each of these
\emph{associated semigroups} is norm continuous.
Moreover,
QS cocycles are characterised
(amongst adapted QS processes with exponential domain)
by the semigroup-decomposition property:
\begin{equation}
\label{eqn: decomp}
X^{f,g}_t =
P^{f(t_0), g(t_0)}_{t_1 - t_0} \cdots P^{f(t_n), g(t_n)}_{t_{n+1} - t_n}
\qquad
(f, g \in \Step, t \in \Rplus)
\end{equation}
in which the set
$\{ 0 = t_0 < t_1 < \cdots < t_n < t_{n+1} = t \}$
contains the points of discontinuity of $f_{[0,t[}$ and $g_{[0,t[}$ (\cite{LW2}, Proposition 3.2).
The \emph{vacuum expectation semigroup}
is the associated semigroup $P^{0,0}$,
and the following conditions on a quasicontractive QS cocycle $X$ are all equivalent:
\begin{rlist}
\item
$X$ is strongly continuous;
\item
$X^*$ (defined below) is strongly continuous;
\item
$X$ is weak operator continuous;
\item
$X$ has strongly continuous expectation semigroup
\end{rlist}
(\cite{LiW}, Lemma 1.2).
Here we must mention an important symmetry of the theory.
Given a QS cocycle $X$,
its \emph{dual cocycle} is defined by
\[
\dualco :=
\big( \rho_t ( X^*_t)\big)_{t \in \Rplus} =
\big( \rho_t ( X_t )^* \big)_{t \in \Rplus},
\]
where $\rho_t:=  \id_{B(\init)} \otol \Fockrho_t$ (\cite{Jou}).
It is easily verified that $\dualco$
is indeed a QS cocycle,
and that
the dual cocycle of $\dualco$ is $X$.
Given a QS cocycle $X$ on $\init$ with noise dimension space $\noise$,
\[
\big( X_{r,t} := \sigma_r (X_{t-r} ) \big)_{0 \les r \les t}
\]
defines a
(\emph{continuous-time}) \emph{bi-adapted covariant} (\emph{left}) \emph{evolution},
that is
\begin{align*}
&
X_{r,t} \in\big(  B(\init) \ot\FockI_{[0,r[} \big) \otol \big(  B( \Fockk_{[r,t[} ) \ot \FockI_{[t,\infty[} \big),
\\
&
X_{r+u, t+u} = \sigma_u ( X_{r, t} ),
\\
&
X_{t,t} = I_{\init\ot\Fockk},
\ \text{ and } \
X_{r,t} =
X_{r,s}
X_{s,t}
\end{align*}
for $r, s, t, u \in \Rplus$
with $r \les s$ and $s \les t$.
Furthermore, every such evolution arises in this way.
Extending the notation~\eqref{eqn: pre Xfg} as follows,
\begin{equation}
\label{eqn: Xfg}
X^{f,g}_{r,t} :=
E^{\ve(f_{[r,t[})} X_{r,t} E_{\ve(g_{[r,t[})},
\end{equation}
the family
$\big( X^{f,g}_{r,t} \big)_{0 \les r \les t}$
forms an evolution in $B(\init)$, for each $f, g \in L^2_{\loc}(\Rplus; \noise)$.
%%%%%%%%%%%%%%%%%%%%%%%%%%%%%%%%%%%%%%%%

\begin{rem}
In this paper we deal with QS left cocycles throughout.
There are also
QS right cocycles, defined in the same way as left cocycles
except that the cocycle identity now reads
$
X_{r+t} = \sigma_r ( X_t ) X_r
$
($r, t \in \Rplus$).
The adjoint and time-reversal operations, given respectively by
$X^* := ( ( X_t )^* )_{t \in \Rplus}$
and
$\reversedco := ( \rho_t (X_t ) )_{t \in \Rplus}$,
turn QS left cocycles into right ones, and vice-versa.
Note that
$\dualco = X^{* {\reversed}} = X^{ {\reversed} *}$.
\end{rem}

For future reference
(in Sections~\ref{section: transformation} and~\ref{section: realisation})
we next discuss the stochastic generation of QS cocycles
and
the associated It\^o algebra of generators, in particular
we give a sample decomposition of the generator of an isometric QS cocycle;
here we are summarising results from~\cite{LQST2},
where further detail may be found.
In Section~\ref{section: transformation}
this is related to the compilation of QRW generators,
and in
Section~\ref{section: realisation}
it is shown how to tailor these compositions/decompositions for
convergence to a given QS cocycle.

As will be increasingly clear,
the crucial
 composition law for (bounded) generators of QS cocycles is the \emph{series product}.
This is the composition on $B(\init \ot \khat )$
defined by
\begin{equation}
	\label{eqn: QIP}
	F_1 \seriesprod F_2 :=
	F_1 + F_2 + F_1 \QIP F_2.
\end{equation}
\noindent
In~\cite{LQST2}
it is shown that
$( B( \init \ot \khat ), \seriesprod, * )$
is a *-monoid,
that is,
an involutive semigroup-with-identity.
A significant representation of this *-monoid is given
in Section~\ref{section: transformation} below.

For operators
$Z \in B(\init)$,
$L \in B(\init; \init \ot \noise)$,
$M \in B(\init \ot \noise; \init)$ and
$W \in B(\init \ot \noise)$, we
set
\begin{align}
\label{eqn: Flft}
&
F_{Z, L, W} :=
\begin{bmatrix}
Z - \tfrac{1}{2} L^*L & -L^*W \\ L & W - I
\end{bmatrix}.
\end{align}
Note that
\begin{align*}
&
( F_{Z, L, W} )^*
\seriesprod
F_{Z, L, W}
=
\begin{bmatrix}
Z^* + Z & 0 \\ 0 & W^*W - I
\end{bmatrix},
\end{align*}
and,
for
$Z_i \in B(\init)$,
$L_i \in B(\init; \init \ot \noise)$,
 and
$W_i \in B(\init \ot \noise)$
($i=1,2$),
\[
F_{Z_1, L_1, W_1} \seriesprod F_{Z_2, L_2, W_2}
=
F_{Z, L, W},
\]
where
\begin{equation}
\begin{aligned}
\label{eqn: Fldecomp}
&
W = W_1W_2,
\\
&
L = L_1 + W_1L_2,
\ \text{ and }
\\
&
Z = Z_1 + Z_2 -\tfrac{1}{2} L_2^* ( I - W_1^* W_1 ) L_2 - i \im L_1^* W_1 L_2.
\end{aligned}
\end{equation}
In particular,
for $Z = Z_0 + \cdots + Z_5$
where
$Z_0, \cdots , Z_5 \in B( \init )$,
\begin{align}
\label{eqn: thus c}
&
F_{Z_0, 0 , I} \seriesprod
F_{Z_1, 0 , I} \seriesprod
F_{Z_2, L , I} \seriesprod
F_{Z_3, 0 , I} \seriesprod
F_{Z_4, 0 , W} \seriesprod
 F_{Z_5, 0 , I}
=
F_{Z, L , W}.
\end{align}

%%%%%%%%%%%%%%%%%%%%%%%%%%%%%%%%%%%%%%%%

For us here, the following properties (all proved in~\cite{LQST2}) are key;
they should be read in conjunction with Theorem~\ref{thm: cocycles} below.

(i)
The isometric structure relation $F^* \seriesprod F =0$
is equivalent to $F$ being of the form $F_{Z, L, W}$
with $Z$ skewadjoint and $W$ isometric,
and
the coisometric structure relation $F \seriesprod F^* = 0$
is equivalent to $F$ being of the form
\[
\begin{bmatrix}
Z - \tfrac{1}{2} M M^* & M \\ -WM^* & W - I
\end{bmatrix},
\]
with $Z$ skewadjoint and $W$ coisometric.

(ii)
For $\beta \in \Real$,
the relations $ F^* \seriesprod F \les 2 \beta \Delta^\perp$
and
$ F \seriesprod F^* \les 2 \beta \Delta^\perp$
are equivalent.

(iii)
Setting $F = F_1 \seriesprod F_2$,
\begin{align*}
&
\text{ if } F_i^* \seriesprod F_i \les 2 \beta_i \QIP^\perp \text{ for } i=1,2,
\text{ then }
 F^* \seriesprod F \les (\beta_1 + \beta_2)\QIP^\perp;
\\
&
\text{ if } F_i^* \seriesprod F_i  = 0  \text{ for } i=1,2,
\text{ then }
 F^* \seriesprod F =0.
\end{align*}

By a \emph{weak solution} of the QS differential equation
$\rd X_t = X_t  \ \! \rd\Lambda_F(t)$
with
$X_0 = I_{\init \ot \Fockk}$,
is meant a family of operators $X = (X_t)_{t\ges 0}$
on $\init \ot \Fockk$
with domain $\init \otul \Exps$
such that,
for all
$f,g \in \Step$, $u,v \in \init$ and $t \in \Rplus$,
\begin{rlist}
\item[(a)]
$
\ip{ u \ve(f) }{ X_t v\ve(g) } =
\ip{ u \ve( f_{[0,t[} ) }{ X_t v\ve( g_{[0,t[} ) } \ip{ \ve( f_{[t,\infty[} ) }{ \ve( g_{[t,\infty[} ) }
$,
\item[(b)]
$
s \mapsto \ip{ u \ve(f) }{ X_s v\ve(g) }
$
is continuous, and
\item[(c)]
$
\ip{ u \ve(f) }{ ( X_t - I_{\init\ot\Fock} ) v\ve(g) }
=
\int_0^t \rd s
\ip{ u \ve(f) }{ X_s E^{\fhat(s)} F E_{\ghat(s)} v\ve(g) }
$.
\end{rlist}
A \emph{strong solution} is a weak solution that
is sufficiently regular that the QS integrals
$\int_0^t X_s  \ \! \rd\Lambda_F(s)$ are defined
and
(c) holds in integrated form:
\begin{itemize}
\item[(c)$'$]
$X_t - I_{\init\ot\Fock} = \int_0^t X_s  \ \! \rd\Lambda_F(s)$
($t\in\Rplus$)
\end{itemize}
(see~\cite{Lbook}).
\emph{Weak regularity} means,
for all
$f,g \in \Step$,
\begin{itemize}
\item[]
$
\big( E^{ \ve(f) } X_t E_{ \ve(g)} \big)_{t\ges 0}
$
is bounded operator valued and locally uniformly bounded.
\end{itemize}

\begin{thm}
[\cite{Fagnola}, Proposition 3.1;~\cite{LW1}, Theorem 7.5]
\label{thm: cocycles}
Let $X$ be an elementary quasicontractive QS cocycle on $\init$ with noise dimension space $\noise$,
and let $\Til$ be a total subset of $\noise$ containing $0$.
Then there is a unique operator $F \in B( \init \ot \khat )$,
called its \emph{stochastic generator}, such that
$X$ weakly satisfies the QS differential equation
\begin{equation}
\label{eqn: QSDE}
X_0 = I_{\init \ot \Fockk}, \quad \rd X_t = X_t  \ \! \rd\Lambda_F(t)
\end{equation}
on the exponential domain $\init \otul \Exps_\Til$.
The QS cocycle
$X$ is strongly continuous and satisfies~\eqref{eqn: QSDE} strongly on $\init \otul \Exps$,
moreover
\begin{align*}
&
 F^* \seriesprod F  \les 2 \beta \QIP^\perp
\text{ if and only if }
( e^{- \beta t } X_t )_{t \ges 0} \text{ is contractive}
\qquad
(\beta \in \Real);
\\
&
 F^* \seriesprod F  = 0
\text{ if and only if }
X \text{ is isometric};
\\
&
F \seriesprod F^*  = 0
\text{ if and only if }
X \text{ is coisometric}.
\end{align*}

Conversely,
let $F \in B(\init \ot \khat )$.
Then the QS differential equation~\eqref{eqn: QSDE} has a unique weakly regular, weak solution,
denoted $X^F$.
Moreover, if $F$
 satisfies $ F^* \seriesprod F  \les 2 \beta \QIP^\perp$
for some $\beta \in \Real$,
then
$X^F$ is an elementary quasicontractive QS cocycle.
\end{thm}

\begin{rems}
(i)
In the converse part,
the QS process $X^F$ need not be bounded if the constraint
$ F^* \seriesprod F  \les 2 \beta \QIP^\perp$
is not imposed.

(ii)
QS generation and duality are related in the following simple way
(\cite{Lbook}, p.~252):
if $X = X^F$ then $\dualco = X^{F^*}$.

(iii)
Suppose that, for $i = 1$, $2$,
$F_i \in B(\init \ot \khat )$ satisfies $ F_i^* \seriesprod F_i  \les 2 \beta_i \QIP^\perp$
for some $\beta_i \in \Real$.
Then
the quasicontractive QS cocycle $X^{F_1 \seriesprod F_2}$ is expressible
in terms of limits of \emph{QS Trotter products} of the cocycles
$X^{F_1}$ and $X^{F_2}$ (\cite{LQST2}, Proposition 3.4).
\end{rems}

%%%%%%%%%%%%%%%%%%%%%%%%%%%%%%%%%%%%%%%%%%%%%%%%%%%%%
% NEW SECTION
%%%%%%%%%%%%%%%%%%%%%%%%%%%%%%%%%%%%%%%%%%%%%%%%%%%%%

\section{Quantum random walks}
\label{section: qrws}

In this section we register the basic facts about quantum random walks on a Hilbert space.
One aim here is to emphasise the very close analogy between quantum random walks and
QS cocycles.
Indeed we show how QS cocycles may naturally be viewed as the continuous-time counterpart to quantum random walks.

For any
$m,n \in \Zplus$ with $m \les n$,
 set
\begin{equation*}
\Toyk_{[m,n[} = \Toyk^\noise_{[m,n[} :=
\khat_{(m)} \ot \cdots \ot \khat_{(n-1)}
\text{ and }
\Toyk_{[n,\infty[} :=
\bigotimes_{p=n}^\infty
\khat_{(p)}
\end{equation*}
where
$ \khat_{(n)} = \khat$ for each $n \in \Zplus$
and the infinite tensor product is with respect to the constant stabilising sequence of unit vectors $\wh{0}$;
also set
\[
 \Toyk := \Toyk_{[0,\infty[}.
\]
Whether intervals are discrete or continuous will always be clear from context.
The `toy Fock space' identifications
\[
 \Toyk =  \Toyk_{[0,m[} \ot  \Toyk_{[m,n[} \ot  \Toyk_{[n,\infty[}
\qquad
(m,n \in \Zplus, m \les n)
\]
are discrete analogues of the continuous tensor decompositions~\eqref{eqn: tensor decomp}
of $\Fockk$.
We use the notation $\ToyI_{[m,n[}$
for the corresponding identity operators.

Two families of endomorphisms of $B(\Toyk)$
are defined by
\[
 \Toysigma_{n}:
T \mapsto \ToyI_{[0,n[} \ot S_n T S_n^*
\ \text{ and } \
\Toyrho_{n}:
T \mapsto  R_n T R_n
\qquad
(n \in \Zplus),
\]
where $S_n$ is the unitary shift operator $\Toyk \to \Toyk_{[n,\infty[}$, and
$R_n := R_{n)} \ot \ToyI_{[n,\infty[}$
for the
selfadjoint unitary operator $R_{n)}$ on $\Toyk_{[0,n[}$
determined by
$R_{n)}  ( \zeta_1 \ot \cdots \ot \zeta_n ) = \zeta_n \ot \cdots \ot \zeta_1$.
Also define the embedding
\[
\Toyj:
B(\khat) \to B(\Toyk), \quad
T \mapsto T \ot \ToyI_{[1,\infty[}.
\]

%%%%%%%%%%%%%%%%%%%%%%%%%%%%%%%%%%%      DEFINITION
\begin{defn}
\label{defn: QRW}
A discrete-time QS (left) cocycle on $\init$ with noise dimension space $\noise$
is a family
$( W_n )_{n \in \Zplus}$ in $B( \init \ot \Toyk )$ such that
\[
W_0 = I_{\init \ot \Toyk},
\ \
W_{l + n} = W_l \,\sigma_l ( W_n )
\ \text{ and } \
W_n \in B( \init \ot \Toyk_{[0,n[} ) \ot \ToyI_{[n,\infty[}
\qquad
(l, n \in \Zplus)
\]
where
$\sigma_l := \id_{B(\init)} \otol \Toysigma_l$.
We refer to these as
(\emph{left}) \emph{quantum random walks} (QRW).
\end{defn}
Thus
QRWs are determined by the family
$\big( W_{n)} \in B( \init \ot \Toyk_{[0,n[} ) \big)_{\!n \in \Nat}$
for which
\begin{equation}
\label{eqn: Wn}
W_n = W_{n)} \ot \ToyI_{[n,\infty[}
\qquad
(n \in \Nat).
\end{equation}

Given a QRW $W$,
the
\emph{dual} QRW is defined by
\[
\dualqrw :=
\big( \rho_n ( W_n^* ) \big)_{n \in \Zplus} =
\big( \rho_n ( W_n )^* \big)_{n \in \Zplus},
\]
where $\rho_n:=  \id_{B(\init)} \otol \Toyrho_n$.
As with QS cocycles,
it is easily verified that $\dualqrw$ is indeed a QRW,
and that its dual is $W$.

Let $G \in B( \init \ot \khat )$.
Then the family
$( W_n )_{n \in \Zplus}$ in $B( \init \ot \Toyk )$
defined by
\[
W_0 = I_{\init \ot \Toyk}
\ \text{ and } \
W_{n} :=
\prod_{0 \les i <  n}^{\longrightarrow}
G_i,
\ \text{ where } \
G_i := \sigma_{i}
\left(
\big( \id_{B(\init)} \otol \Toyj \big) (G)
\right)
\ \text{ for } \ n \in \Nat
\]
is readily seen to define a QRW on $\init$,
which is denoted $W^G$,
and, since
$
B( \init \ot \Toyk_{[0,1[} ) \ot \ToyI_{[1,\infty[}
= \Ran \big( \id_{B(\init)} \otol \Toyj \big)
$,
it is clear that
every QRW arises in this way.
 The operator
$G$ is referred to as the \emph{generator} of the QRW.
Generation and duality are related in the following simple way:
if $W = W^G$
then
$\dualqrw = W^{G^*}$.

Given a left QRW $W$ on $\init$ with noise dimension space $\noise$,
\[
\big( W_{l, n} := \sigma_l ( W_{n-l} ) \big)_{0 \les l \les n}
\]
defines a \emph{discrete-time bi-adapted covariant evolution}, that is
\begin{align*}
&
W_{l,n} \in
\big( B(\init) \ot\ToyI_{[0,l[} \big) \, \otol \big( B( \Toyk_{[l,n[} ) \ot \ToyI_{[n,\infty[} \big),
\\
&
W_{l+p, n+p} = \sigma_p ( W_{l, n} ),
\\
&
W_{n,n} = I_{\init\ot\Toyk}
\ \text{ and } \
W_{l,n} =
W_{l,m}
W_{m,n},
\end{align*}
for $l, m, n, p \in \Zplus$
with $l \les m$ and $m \les n$.
Conversely,
every such evolution
$( W_{l,n} )_{0 \les l \les n}$
is so determined by the left QRW
$( W_{0,n} )_{n \in \Zplus}$.
In view of the covariance property,
\begin{equation*}
%\label{eqn: prod}
W_{l,n} =
\prod_{l \les m <  n}^{\longrightarrow}
W_{m,m+1},
\end{equation*}
and, in terms of its generator $G$,
\[
W_{m, m+1} =
\big(
 \id_{B(\init)} \otol
 ( \Toysigma_{m} \circ \Toyj )
 \big) (G)
\qquad
(m \in \Zplus).
\]

\begin{rem}
Here, as for QS cocycles,
 we deal with left QRWs throughout.
There are also
right QRWs, defined in the same way as left QRWs
except that the cocycle identity is switched to
$
W_{l+n} = \sigma_l ( W_n ) W_l
$
($l, n \in \Zplus$).
The adjoint and time-reversal operations,
given respectively by
$W^* := ( ( W_n )^* )_{n \in \Zplus}$
and
$W^{\reversed} := ( \rho_n (W_n ) )_{n \in \Zplus}$,
turn left QRWs into right ones, and vice-versa.
Note that,
for a left or right QRW,
$\dualqrw = W^{* \reversed} = W^{ \reversed *}$.
\end{rem}
%%%%%%%%%%%%%%%%%%%%%%%%%%%%%%%%%%%%%%%%%%%%%
%%%%%%%%%%%%%%%%%%%%%%%%%%%%%%%%%%%%%%%%%%%%%

\section{Embedding}
\label{section: embed QRW}

Suitably scaled discrete QS cocycles converge to continuous QS cocycles in the sense made precise in
Theorem~\ref{thm: main} below.
This entails embedding QRWs into the habitat of continuous-time processes,
for which the relevant definition follows.

%%%%%%%%%%%%%%%%%%%%%%%%%%%%%%%%%%%      DEFINITION
\begin{defn}
Let $h > 0$.
The \emph{$h$-scale embedded left
QRW generated by $G \in B(\init \ot \khat)$},
is the bounded-operator QS process $X$ on $\init$,
with noise dimension space $\noise$,
defined by
$X_t := X_{0, h \lfloor t/h \rfloor}$
 where
\[
X_{hl, hn} :=
\prod_{
l \les m < n
}^{\longrightarrow}
X_{hm, h(m+1)}
\qquad (l,n \in \Zplus)
\]
and
\[
X_{hm, h(m+1)}
 :=
\big(
 \id_{B(\init)} \otol
 ( \Focksigma_{hm} \circ \Fockj_h )
 \big) (G)
\qquad
(m\in\Zplus),
\]
through
the
embedding
\[
\Fockj_h:
B(\khat) \to B(\Fockk), \quad
T \mapsto J_h\,  T \, (J_h)^*  \ot \FockI_{[h,\infty[}
\]
in which
$J_h: \khat \to \Fockk_{[0,h[}$
denotes the isometry determined by the prescription
\[
\chat \mapsto \vetrunc ( h^{-1/2} c ).
\]
Here the vector
$ h^{-1/2} c $
is
considered as the corresponding constant function in
$L^2([0,h[; \noise)$, and the following
\emph{truncated exponential vectors} are employed
\begin{equation}
\label{eqn: truncated}
\vetrunc(g) := (1, g, 0, 0, \cdots )
\quad
(g \in L^2(I; \noise), I \text{ a subinterval of } \Rplus).
\end{equation}
\end{defn}

\begin{notn}
The $h$-scale embedded left
QRW generated by $G \in B(\init \ot \khat)$
is denoted
$\hXG$.
\end{notn}

%%%%%%%%%%%%%%%%%%%%%%%%%%%%%%%%%%%      REMARK
\begin{rem}
For future reference,
we note the following elementary estimate
on embedded quantum random walks:
\begin{equation}
\label{eqn: elem est}
\norm{ \hXG_t }
\les
\norm{ G }^{\lfloor t/h \rfloor}
\qquad
(t \in \Rplus).
\end{equation}
In particular,
the process $\hXG$ is contractive if the QRW generator $G$ is.
It is obviously isometric or coisometric if and only if $G$ has the same property.
\end{rem}

%%%%%%%%%%%%%%%%%%%%%%%%%%%%%%%%%%%
%%%%%%%%%%%%%%%%%%%%%%%%%%%%%%%%%%%%%%%%%%%% SECTION
%%%%%%%%%%%%%%%%%%%%%%%%%%%%%%%%%%%%%%%%%%%%

 \section{Convergence}
\label{section: standard qrw approx}

In this section we show that
suitably scaled families of QRWs converge to QS cocycles,
in analogy with the Donsker invariance principle.

For $n \in \Zplus$, and for
$g$ in either $L^2_{\loc}(\Rplus; \noise)$ or
$L^2([hn, h(n+1)[; \noise)$,
let
$g[n,h]$ denote
the average of $g$ over the interval $[hn, h(n+1)[$:
\begin{equation}
\label{eqn: average}
g[n,h] :=
 h^{-1} \int_{hn}^{h(n+1)} g.
\end{equation}
Thus,
for $g \in L^2([0,h[; \noise)$,
$
( J_h )^* \ve(g) = \wh{\sqrt{h}\, g[0,h]}.
$

%%%%%%%%%%%%%%%%%%%%%%%%%%%%%%%%%%%     REMARK
\begin{rem}
Observe that, in
the notation
\begin{equation}
\label{eqn: discr evol}
X^{f,g}_{mh,nh} :=
E^{\ve(f_{[hm, hn[})}
X_{mh,nh}
E_{\ve(g_{[hm, hn[})}
\quad
(f,g\in \Step,
m,n\in\Zplus, m\les n),
\end{equation}
where $X = \hXG$,
we have discrete evolutions
for each $f, g \in \Step$:
\begin{equation}
\label{eqn: d evol}
X^{f,g}_{hn,hn} = I_\init, \quad
X^{f,g}_{hl,hm} \, X^{f,g}_{hm,hn} =
X^{f,g}_{hl,hn}
\quad
(l, m,n \in \Zplus, l \les m \les n).
\end{equation}
\end{rem}

For $h > 0$, define the standard scaling matrix
(\emph{cf.}~\cite{LiP})
\begin{equation*}
%\label{eqn: scaling matrix}
\mathcal{S}_h^\noise :=
\begin{bmatrix} h^{-1/2} &0  \\ 0 & I_\noise \end{bmatrix}
\in B(\khat),
\end{equation*}
and let
$\mathit{s}_h$ denote conjugation by $I_\init \ot \mathcal{S}_h^\noise$ on
 $B(\init\ot\khat)$,
 thus
 \begin{equation}
 \label{eqn: scaling}
 \mathit{s}_h
 \left(
  \begin{bmatrix}  A &   C  \\   B & D \end{bmatrix}
 \right)
 =
 \begin{bmatrix} h^{-1} A &  h^{-1/2} C  \\  h^{-1/2} B & D \end{bmatrix}
\qquad
(h>0).
 \end{equation}
 On the one hand
 the scaling is motivated by purely (quantum) probabilistic considerations
 via Donsker's functional central limit theorem,
 and on the other hand
 it is related to the weak coupling and low density limits of statistical physics
 (\cite{vanHov},~\cite{Davies},~\cite{Dum}).
 The connection is emphasised in~\cite{AFL}, for example;
 for further detailed discussion on this, see~\cite{AJrepeated}.
 In Section~\ref{section: repeated quantum interactions}
 we see how the scaling operates in the important example
 of repeated quantum interactions.

%%%%%%%%%%%%%%%%%%%%%%%%%%%%%%%%%%%      LEMMA
\begin{lemma}
\label{lemma: evol to Euler}
Set
$X = \hXG$
where $G \in B(\init \ot \khat)$ and $h > 0$.
Let $f$, $g\in\Step$
and $m, n \in \Zplus$ with $m \les n$.
\begin{alist}
\item
Then
\begin{equation}
\label{eqn: id 1}
X^{f,g}_{hn,h(n+1)} =
I_\init +
h \,
E^{\wh{f[n,h]}}\,
 \mathit{s}_h (G - \QIP^\perp) \,
E_{\wh{g[n,h]}}
\end{equation}
and
\begin{equation}
\label{eqn: id 2}
\big\|
X^{f,g}_{hn,h(n+1)} -
I_\init \big\|
\les
h \,
\max_{c \in \Ran f, d \in \Ran g}
\big\|
E^{\chat}\,
 \mathit{s}_h (G - \QIP^\perp) \,
E_{\dhat}
\big\|.
\end{equation}
\item
Suppose that $f$ and $g$ are constant,
with values $c$ and $d$ respectively,
on the interval $[hm, hn[$.
Then
\begin{equation}
\label{eqn: id 3}
X^{f,g}_{hm,hn} =
\big(
I_\init +
h \,
E^{\chat}\,
 \mathit{s}_h (G - \QIP^\perp) \,
E_{\dhat}
\big)^{n-m}.
\end{equation}
\end{alist}
\end{lemma}

%%%%%%%%%%%%%%%%%%%%%%%%%%%%%%%%%%%      proof of LEMMA
\begin{proof}
(a)
Since
$\wh{\sqrt{h} c} = \sqrt{h}\, \mathcal{S}_h^\noise  \chat$
for $c \in \noise$,
the
 first identity follows from the definition:
\begin{align*}
X^{f,g}_{hn,h(n+1)} -
I_\init
&=
E^{\wh{\sqrt{h}\, f[n,h]}}\,
 (G - \QIP^\perp) \,
E_{\wh{\sqrt{h}\, g[n,h]}}
\\
&=
h \,
E^{\wh{f[n,h]}}\,
 \mathit{s}_h (G - \QIP^\perp) \,
E_{\wh{g[n,h]}}.
\end{align*}
Since
\[
\wh{f[n,h]} =  h^{-1} \int_{hn}^{h(n+1)} \fhat
\in
\Conv \Ran \fhat,
\]
and similarly for $g$,~\eqref{eqn: id 2}
follows from~\eqref{eqn: id 1}.

(b)
Since
$\wh{f[j,h]} = \chat$ and $\wh{g[j,h]} = \dhat$
for $j \in \{ m, \ldots, n-1 \}$,
this follows from
 the factorisation
\[
X^{f,g}_{hm,hn} :=
X^{f,g}_{hm,h(m+1)} \cdots X^{f,g}_{h(n-1),hn}
\]
and identity~\eqref{eqn: id 1}.
\end{proof}

%%%%%%%%%%%%%%%%%%%%%%%%%%%%%%%%%%%      REMARKS

In order to obtain the approximation result below
in its proper form, we need a lemma.

%%%%%%%%%%%%%%%%%%%%%%%%%%%%%%%%%%%      LEMMA
\begin{lemma}
\label{lemma: weak to strong}
For a Hilbert space $\Hil$ and compact subinterval $J$ of $\Rplus$,
let $(a_\lambda)_{\lambda \in \Lambda}$ be a net of contraction-operator-valued maps
from $J$ to $B(\Hil)$,
let $a: J\to B(\Hil)$ be isometry valued and strongly continuous,
and suppose that
$\ip{\zeta}{a_\lambda( \cdot ) \eta} \to \ip{\zeta}{a( \cdot ) \eta}$
uniformly,
for all $\zeta$, $\eta \in \Hil$.
Then
$a_\lambda( \cdot ) \eta \to a( \cdot ) \eta$
uniformly,
for all $\eta \in \Hil$.
\end{lemma}

%%%%%%%%%%%%%%%%%%%%%%%%%%%%%%%%%%%      proof of LEMMA
\begin{proof}
Let $\eta \in \Hil$ and $\epsilon > 0$.
Since $a$ is strongly continuous and $J$ is compact,
there is an $\Hil$-valued step function
$\varphi = \sum_{j=1}^N \zeta_j \indset_{J_j}$
such that
$\sup_{t \in J} \| a(t) \eta - \varphi(t) \| < \epsilon$.
Therefore,
for all $t \in J$,
\begin{align*}
 \| ( a_\lambda(t) - &a(t) ) \eta \|^2
\\
&\les
2 \re
\ip{a(t)  \eta}{ ( a(t) - a_\lambda(t) ) \eta}
\\
&=
2 \re
\ip{a(t)  \eta - \varphi(t)}{ ( a(t) - a_\lambda(t) ) \eta}
+
\sum_{j=1}^N \indset_{J_j}(t)  \ip{\zeta_j}{ ( a(t) - a_\lambda(t) ) \eta}
\\
&=
4 \| \eta \| \epsilon +
\max_{j=1}^N | \ip{\zeta_j}{ ( a(t) - a_\lambda(t) ) \eta}|.
\end{align*}
Since the second term tends to zero uniformly, the result follows.
\end{proof}

In the proof of Theorem~\ref{thm: main} below,
we use Euler's exponential formula in the following form. Let $a$,
$a(h) \in B(\init)$, for $h > 0$, and let $T \in \Rplus$;
if $a(h) \to a$ as $h \to 0$ then
\begin{equation}
\label{eqn: euler}
\sup_{ [r,t] \subset [0,T] }
\big\|
( I_\init + h a(h) )^{ \lfloor t/h \rfloor -  \lfloor r/h \rfloor } -
e^{ (t-r) a }
\big\|
\to 0
\ \text{ as } \ h \to 0.
\end{equation}

%%%%%%%%%%%%%%%%%%%%%%%%%%%%%%%%%%%      THEOREM
\begin{thm}
\label{thm: main}
Let
 $\Til'$ and $\Til$ be total subsets of $\noise$ containing $0$,
and let
$F$, $G_h \in B(\init\ot\khat)$ \tu{(}$h > 0$\tu{)} satisfy
\begin{equation}
\label{eqn: QRW to QSC structure}
E^{\chat}
\big[
\mathit{s}_h ( G_h - I_{\init \ot \khat} ) - F )
\big]
E_{\dhat}
\to 0 \ \text{ as } \ h \to 0
\qquad
(c \in \Til', d \in \Til).
\end{equation}
Then
\begin{equation}
\label{eqn: discrete approx}
\sup_{t \in [0,T]}
\big\|
E^{\ve'}
\big(
\hXGh_t - X^F_t
 \big) E_\ve
 \big\|
\to 0
\text{ as }
h \to 0
\qquad
(\ve' \in \Exps_{\Til'}, \ve \in \Exps_{\Til}, T \in \Rplus).
\end{equation}
Moreover,
the following refinements hold.
\begin{alist}
\item
Suppose that, for some  $\beta \in \Real$ and all $T \in \Rplus$,
\[
 \sup_{h > 0,\, t \in  [0,T] }
\norm{ G_h }^{ \lfloor t/h \rfloor }
< \infty
\
\text{ and }
\
F^* \seriesprod F \les 2 \beta \QIP^\perp.
\]
Then~\eqref{eqn: discrete approx} may be strengthened to
\[
\sup_{t\in [0,T] }
\big\|
(\id_{B(\init)} \otol \varphi )
\big(
\hXGh_t - X^F_t
 \big)
 \big\|
\to 0
\text{ as }
h \to 0
\qquad
(\varphi \in B(\Fockk)_*, T \in \Rplus).
\]
\item
Suppose that each $G_h$ is a contraction and $F$ satisfies
$F^* \seriesprod F = 0$.
Then also
\[
\sup_{t\in [0,T] }
\big\|
\big(
\hXGh_t - X^F_t
 \big)\xi
 \big\|
\to 0
\text{ as }
h \to 0
\qquad
(\xi \in \init \ot \Fockk, T \in \Rplus).
\]
\item
Suppose that
 each $G_h$ is a contraction and $F$ satisfies
$F \seriesprod F^* = 0$.
Then also
\[
\sup_{t\in [0,T] }
\big\|
\big(
\hXGh_t - X^F_t
 \big)^*
\xi
 \big\|
\to 0
\text{ as }
h \to 0
\qquad
(\xi \in \init \ot \Fockk, T \in \Rplus).
\]
\end{alist}
\end{thm}

%%%%%%%%%%%%%%%%%%%%%%%%%%%%%%%%%%%      proof of THEOREM
\begin{proof}
Fix $T \in \Rplus$ and set
$\Xsofth := \hXGh$ and $X := X^F$.
The first part amounts to fixing
$f \in \Step_{\Til'}$ and $g \in \Step_{\Til}$,
and showing that
\[
E^{ \ve( f_{[0,t[} ) } \big( X^{(h)}_t - X_t \big) E_{ \ve( g_{[0,t[} ) }
\to
0
\text{ uniformly on }
[0,T],
\text{ as }
h \to 0.
\]
Fix
$f$ and $g$ accordingly,
and set
\begin{align*}
&
Q^{(h)}_{hm, hn}
:=
E^{ \ve( f_{[hm, hn[} ) }  X^{(h)}_{hm, hn}  E_{ \ve( g_{[hm, hn[} ) }
\qquad
(m,n \in \Zplus, m<n)
\\
&
Q^{(h)}_t :=
E^{ \ve( f_{[0,t[} ) }  X^{(h)}_t  E_{ \ve( g_{[0,t[} ) }
\ \text{ and } \
Q_t :=
E^{ \ve( f_{[0,t[} ) }  X_t  E_{ \ve( g_{[0,t[} ) }
\qquad
(t>0).
\end{align*}
Choose
$T_+ > \max ( D \cup \{T\} )$
where
$D$ is the union of the sets of points of discontinuity of $f$ and $g$,
let
\[
\{ t_0 < \cdots < t_{N+1} \}
=
\{0 \} \cup D \cup \{ T_+\}.
\]
Henceforth $h>0$ is assumed to be smaller than $\mesh D$.
By the discrete evolution property~\eqref{eqn: d evol},
for $t > 0$,
\[
 Q^{(h)}_{t} =
\sum_{k=0}^N
\indset_{[t_k,t_{k+1}[}(t) \,
\prod_{1 \les j <  k}^{\longrightarrow}
\!
A_j(h)
\
B_k(h,t)
 C( h \lfloor t/h \rfloor, t )
\]
where,
when
$j \in \{ 1, \cdots , N-1 \}$,
\[
A_j(h) :=
 \Qsofth_{ h \lfloor t_j/h \rfloor, h( 1 + \lfloor t_j/h \rfloor ) }
 \Qsofth_{ h( 1 + \lfloor t_j/h \rfloor ),  h \lfloor t_{j+1}/h \rfloor }
\]
and,
when
$k \in \{ 0, \cdots , N \}$,
and
$t \in [t_k, t_{k+1}[$,
\[
B_k(h,t) =
 \left\{
  \begin{array}{ll}
  \Qsofth_{ h \lfloor t_k/h \rfloor, h( 1 + \lfloor t_k/h \rfloor ) }
 \Qsofth_{ h( 1 + \lfloor t_k/h \rfloor ),  h \lfloor t/h \rfloor }
  &
 \quad
 \text{if $ \lfloor t_k/h \rfloor < \lfloor t/h \rfloor$},
     \\
I_\init
   &
 \quad
\text{if $ \lfloor t_k/h \rfloor = \lfloor t/h \rfloor$},
  \end{array} \right.
\]
for the operators defined by
$
C(u,v) :=  \ip{\ve(f_{[u,v[})}{\ve(g_{[u,v[})} I_\init
$
($0 \les u \les v$).
On the other hand,
by the semigroup decomposition of QS cocycles,
\[
 Q_{t}=
\sum_{k=0}^N
\indset_{[t_k,t_{k+1}[}(t) \,
P^{(0)}_{t_1 - t_0}
\cdots
P^{(k-1)}_{t_k - t_{k-1}}
 P^{(k)}_{t - t_k},
\]
where,
 for $i = 0$, \ldots, $N$,
$P^{(i)}$ denotes the $(f(t_i), g(t_i))$-associated semigroup of
the QS cocycle $X$,
defined in~\eqref{eqn: assoc semi}.

Now set
\begin{align}
& \label{eqn: Fh}
F_h :=
\hscale \big( G_h - I \big) = \hscale \big( G_h - \Delta^\perp \big) - \Delta,
\ \text{ and }
\\
& \nonumber
M_h := \max \big\{ \norm{ E^{\chat} ( F_h + \Delta ) E_{\dhat} }: c \in \Ran f, d \in \Ran g \big\}
\qquad
(h>0).
\end{align}
Then,
since
$E^{\chat} ( F_h + \Delta ) E_{\dhat} \to  E^{\chat} ( F + \Delta ) E_{\dhat}$
as $h \to 0$, for all $c \in \Til'$ and $d \in \Til$,
$\limsup_{h \to 0} M_h < \infty$.
Lemma~\ref{lemma: evol to Euler} implies that,
for all $n \in \{ 0, \cdots , N \}$ and $h > 0$,
\begin{align*}
\big\|
 \Qsofth_{ h \lfloor t_n/h \rfloor, h( 1 + \lfloor t_n/h \rfloor ) } - I_\init
\big\|
=
h
\,
\big\|
E^{\fhat[n,h]} \big( F_h + \Delta \big) E_{\ghat[n,h]}
\big\|
\les
h M_h,
\end{align*}
since $\fhat[n, h] \in \Conv \Ran \fhat$, and similarly for $g$.
Also,
since
$\big\|  C( h \lfloor t/h \rfloor, t ) - I_\init \big\| \les h \exp \norm{f} \norm{g}$,
\[
  C( h \lfloor t/h \rfloor, t ) \to I_\init
\text{ uniformly in } t \text{ as } h \to 0.
\]
Therefore,
since the generator of $P^{(j)}$ is
$E^{\fhat(t_j)} \big( F + \Delta \big) E_{\ghat(t_j)}$,
Lemma~\ref{lemma: evol to Euler} and Euler's formula imply that, as $h \to 0$
\[
A_j(h) \to P^{(j)}_{t_{j+1} - t_j}
\ \text{ and } \
\sup_{ t \in [t_k, t_{k+1}[ }
\big\|
B_k(h,t) - P^{(k)}_{t  - t_k}
\big\|
\to 0
\]
for $j \in \{ 0, \cdots , N-1 \}$ and $k \in \{ 0, \cdots , N \}$.
It follows that $ Q^{(h)}_{0,t} \to  Q_{t}$ uniformly on $[0,T]$, as required.

(a)
By the basic estimate~\eqref{eqn: elem est},
 $\{ \Xsofth_t : h > 0, t \in [0,T] \}$ is uniformly bounded
 and,
by the characterisation of quasicontractivity
of elementary QS cocycles
 recalled in Theorem~\ref{thm: cocycles},
 $\norm{ X_t } \les e^{ \beta t}$ ($t \in \Rplus$)
 so  $\{ X_t :  t \in [0,T] \}$ is uniformly bounded too.
The result therefore follows from the first part,
by the norm totality of the family
$\{ \omega_{\ve',\ve}: \ve' \in \Exps_{\Til'}', \ve \in \Exps_{\Til}  \}$
in $B(\Fockk)_*$
and the well-known
 fact (\emph{e.g.}~\cite{EfR}, Corollary 2.2.3) that
$\cbnorm{ \varphi } = \norm{ \varphi}$
for any
$\varphi \in B( \Fockk )_*$.

(b \& c)
It follows from (a) that
\[
\sup_{t\in [0,T]}
\big|
\ip{\zeta}{ (
\Xsofth_t - X_t
 ) \eta}
 \big|
\to 0
\text{ as }
h \to 0
\qquad
(\zeta, \eta \in \init \ot \Fockk).
\]
By Theorem~\ref{thm: cocycles} and Remark (ii) following it,
$X$ and $X^*$ are both strongly continuous.
Since $X$ is isometric if $F^* \seriesprod F = 0$ and
$X^*$ is coisometric if $F \seriesprod F^* = 0$,
(b) and (c) follow from Lemma~\ref{lemma: weak to strong}.
\end{proof}

%%%%%%%%%%%%%%%%%%%%%%%%%%%%%%%%%%%      REMARKS
\begin{rems}
(i)
A useful generalisation arises from the introduction of an extra parameter from a directed set $\Lambda$.
Thus,
for a net $( G_{ h, \lambda } )_{h>0, \lambda \in \Lambda}$ in $B( \init \otimes \khat )$,
Theorem~\ref{thm: main} holds with respect to the net of processes
$( X^{h,G_{ h, \lambda }} )_{h>0, \lambda \in \Lambda}$.
This is exploited (with $\Lambda = \Nat$) in Theorem~\ref{thm: Ghn} below.

(ii)
The above theorem
may be derived from~Theorem 7.6 of~\cite{BQRW},
and corresponds to~Theorem 5.1 of~\cite{DaL}.
Theorem 13 of~\cite{AtP} is a version of this result, established
under hypotheses which are much stronger when $\noise$ is infinite dimensional.
For an earlier form,
see~\cite{Par}, Theorem 4.1.
Here,
 by
direct application of the semigroup decomposition of QS cocycles,
we avoid
 the approximation of quantum Wiener integrals by their discrete analogues;
 we also avoid any need to
appeal to vacuum-adapted QS calculus
or sesquilinear QS calculus,
thereby achieving a much simpler and more direct proof.
Furthermore, unlike in~\cite{DaL},
there is
no restriction to dyadic rational discretisation.

(iii)
Theorem~\ref{thm: main}
 may also be profitably viewed in terms of the approximation of
\emph{elementary evolutions}, in the sense of~\cite{DL1}, by discrete evolutions.
More specifically,
Lemma~\ref{lemma: evol to Euler} yields the alternative representation
\[
\Qsofth_{t} = C^h_{t}
\prod_{0 \les n < \lfloor t/h \rfloor}^{\longrightarrow}
\big( I_\init + h \, a ( n, h ) \big)
\qquad
( t \in \Rplus ),
\]
where
$C^h_{t} = \ip{ \ve ( f_J )}{\ve ( g_J )}$
for some subset $J$ of
$\big[ h( \lfloor t/h \rfloor - 1 ), t \big[$
such that $| J | \les 2 h$,
and
\[
a ( n, h ):=
E^{\wh{f[n, h]}} \big( F_h + \Delta \big) E_{\wh{g[n, h]}}
\qquad
(n \in \Zplus),
\]
where $F_h$ is given by~\eqref{eqn: Fh}.
On the other hand,
$E = \big( E_{r,t} \big)_{0 \les r \les t}$ is an elementary evolution whose generator takes the form
\[
a: s \mapsto E^{\fhat(s)} \big( F + \Delta \big) E_{\ghat(s)},
\]
and
Theorem 3.2
of~\cite{DL1}
implies that,
in the notation~\eqref{eqn: average},
\[
\prod_{0 \les n < \lfloor t/h \rfloor}^{\longrightarrow}
\big( I_\init + h \, a [ n, h ] \big)
\to
Q_{t}
\]
uniformly for $0 \les t \les T$ ($T \in \Rplus$).
The main part of Theorem~\ref{thm: main} then follows from the fact that
\[
\max_{n \in \Nat}
\Big\{
\big\|
a ( n, h ) - a [ n, h ]
\big\|:
 f \text{ and } g \text{ are constant on } [hn, h(n+1)[
\Big\}
\to 0
\ \text{ as } \ h \to 0,
\]
and
\[
\#
\Big\{
n \in \Nat:
 f \text{ or } g \text{ is not constant on } [hn, h(n+1)[
\Big\}
\les
\#
D
\qquad
(0 < h < \mesh D),
\]
where $D$ is the union of the sets of points of discontinuity of $f$ and $g$.

(iv)
The strategy of employing semigroup decomposition in the proof of Theorem~\ref{thm: main}
adapts well from operator cocycles
to
QS mapping cocycles,
and more generally
 to Banach-algebra-valued QS sesquilinear cocycles;
specifically, Theorem 3.6 of~\cite{DaL} may be proved directly,
along the lines of Remark (iii) above.

(v)
The basic hypothesis of Theorem~\ref{thm: main}
is equivalent to the condition
\begin{equation*}
E^{\chat}
\big[
\mathit{s}_h ( G_h^* - I ) - F^* )
\big]
E_{\dhat}
\to 0 \ \text{ as } \ h \to 0
\qquad
(c \in \Til, d \in \Til'),
\end{equation*}
with corresponding refinements.
Therefore, the theorem also yields convergence of the embedded QRWs
$( X^{h,G_h^*} )_{h>0}$ to $X^{F^*} = \dualco$, the dual cocycle of $X$.

(vi)
The sets
 $\Til'$ and  $\Til$ typically each consist of vectors from an
 orthonormal basis for $\noise$ augmented by the vector $0$.

(vii)
In (a) the limit QS cocycle $X^F$ is quasicontractive,
with
$
\norm{ X^F_t } \les e^{ \beta t }
$;
in (b) the cocycle
$X^F$ is isometric, and is unitary if also
$
F \seriesprod F^* = 0
$.
These follow from
 the characterisations of
quasicontractivity, isometry and unitarity of
elementary QS cocycles listed in Theorem~\ref{thm: cocycles}.

(viii)
The basic hypothesis~\eqref{eqn: QRW to QSC structure}
is usefully expressed in the following equivalent form:
\begin{equation*}
%\label{eqn: again QRW to QSC structure}
E^{\chat}
\big[
\mathit{s}_h ( G_h - \QIP^\perp ) - ( F + \QIP )
\big]
E_{\dhat}
\to 0 \ \text{ as } \ h \to 0
\qquad
(c \in \Til', d \in \Til).
\end{equation*}
Then, writing in the block-matrix forms
\begin{equation}
\label{eqn: block Gh and F}
G_h =
\begin{bmatrix} I_{\init} +  h K_h & \sqrt{h} M_h \\ \sqrt{h} L_h &  C_h  \end{bmatrix}
\ \text{ and } \
F =
\begin{bmatrix}  K & M \\ L &  C - I  \end{bmatrix},
\end{equation}
it follows that
\[
\mathit{s}_h ( G_h - \QIP^\perp ) =
\begin{bmatrix}  K_h &  M_h \\  L_h &  C_h  \end{bmatrix}
\ \text{ and } \
F + \QIP =
\begin{bmatrix}  K & M \\ L &  C  \end{bmatrix} .
\]
In these terms~\eqref{eqn: QRW to QSC structure} amounts to the following
more transparent condition:
\begin{equation}
\label{eqn: block QRW to QSC structure}
E^{\chat}
\begin{bmatrix}  K_h- K & M_h - M \\   L_h - L &  C_h - C  \end{bmatrix}
E_{\dhat} \
\to 0 \ \text{ as } \ h \to 0
\qquad
(c \in \Til', d \in \Til).
\end{equation}
When
$\dim \noise < \infty$,
this is equivalent to the simple norm-convergence conditions
\[
 K_h \to K, \  L_h \to  L, \  M_h \to M \  \text{ and } \  C_h \to C
 \ \text{ as } \ h \to 0.
\]
However, when $\noise$ is infinite dimensional,
 it is
only the \emph{components} of
$L_h$, $M_h$ and $C_h$
with respect to some total families $\Til$ and $\Til'$ in $\noise$
(such as orthogonal bases)
that need to converge to the corresponding components of
$L$, $M$ and $C$.

(ix)
The theorem begs two questions.
The first is,
given a generator $F$ of a QS cocycle,
can such a family of operators $( G_h )_{h > 0}$ be found?
This is resolved in Theorem~\ref{thm: strong cgce}
%Section~\ref{section: realisation}
where it is shown that it easily can,
and that moreover, the operators may be chosen to be respectively
isometric, coisometric, or unitary
if
the QS cocycle generated by $F$ has that property.
The second question is,
what families of, say unitary, operators $( G_h )_{h > 0}$
have scaled limits $F$, in the sense of~\eqref{eqn: QRW to QSC structure},
which generate unitary QS cocycles?
In Section~\ref{section: repeated quantum interactions},
the repeated quantum interaction model is shown to provide
a wide source of examples of such families.
\end{rems}

The next example is instructive.

\begin{example}[Preservation-type QS cocycles]
\label{eg: pres}
Let $F = \left[\begin{smallmatrix} 0 & 0 \\ 0 & C - I \end{smallmatrix} \right] \in B(\init \ot \khat)$,
and
$G = F + I = \left[\begin{smallmatrix} I & 0 \\ 0 & C \end{smallmatrix} \right] \in B(\init \ot \khat)$,
for a contraction $C \in B( \init \ot \noise )$.
Note the identification
\[
\init \ot \Fock^\noise =
\bigoplus_{n \in \Zplus} \Big( \init \ot \noise^{\ot n} \ot \Fock^{(n)}_{[0,t[} \Big) \ot \Fockk^\noise_{[t, \infty[}
\qquad
(t \in \Rplus),
\]
in which
$\bigoplus_{n \in \Zplus} \Fock^{(n)}_{[0,t[}$ is the chaos decomposition of $\Fock^{\Comp}_{[0,t[}$.

On the one hand
the QS cocycle $X^F$ has the explicit description
(\cite{Lbook}, Example 5.3)
\[
X^F_t =
\bigoplus_{n \in \Zplus} \Big( \wt{G}_n \ot I^{(n)}_{[0,t[} \Big) \ot I^{\Fock^\noise}_{[t, \infty[}
\qquad
( t \in \Rplus ),
\]
where $I^{(n)}_{[0,t[}$ denotes the identity operator on $\Fock^{(n)}_{[0,t[}$
and,
in the notation~\eqref{eqn: Wn},
\[
\wt{G}_n :=
\big( I_\init \ot J^{\ot n} \big)^*
W^G_{n)}
\big( I_\init \ot J^{\ot n} \big)
\qquad
(n \in \Zplus)
\]
in which $J$ denotes the natural isometry $\noise \to \khat$.
On the other hand, since in this case
\[
F = \hscale ( G - I )
\ \text{ for all } h > 0,
\]
Theorem~\ref{thm: main} implies that
\[
\sup_{t\in [0,T] }
\big\|
(\id_{B(\init)} \otol \varphi )
\big(
X^{h,G_t} - X^F_t
 \big)
 \big\|
\to 0
\text{ as }
h \to 0
\qquad
(\varphi \in B(\Fock^\noise)_*, T \in \Rplus).
\]
\end{example}

\begin{rem}
In the pure-noise case, when the initial space is just $\Comp$,
 $X^F$ is given by the formula
\[
X^F_t = \Gamma \big( M_{[0,t[} \ot C +  M_{[t, \infty[} \ot I_\noise \big)
\qquad
( t \in \Rplus )
\]
where, for any subinterval $J$ of $\Rplus$,
$M_J \in B( L^2(\Rplus))$ denotes the multiplication operator by the indicator function of $J$,
and
$\Gamma$ denotes the second quantisation operation defined in~\eqref{eqn: 2ndQ}.
\end{rem}

%%%%%%%%%%%%%%%%%%%%%%%%%%%%%%%%%%%%%%%%%%%%%%%%%
As a fast corollary to Theorem~\ref{thm: main}
we obtain (pace Remark (iv) above)
a basic QRW approximation result for inner QS flows on a full operator algebra.

\begin{cor}
\label{cor: inner}
Let $j$ be the QS flow on $B(\init)$ induced by an elementary
unitary QS cocycle $X$ on $\init$ with noise dimension space $\noise$,
thus
\[
j_t(x) = X_t \big( x \ot I_{\init \ot \Fockk} \big) X^*_t
\qquad
( x \in B(\init), t \in \Rplus ),
\]
and suppose that $X$ has stochastic generator $F$.
For $h > 0$, let $j^{(h)}$ be the mapping process on $B(\init)$
given by
\[
j^{(h)}_t(x) = \hXGh_t \big( x \ot I_{\init \ot \Fockk} \big) ( \hXGh_t )^*
\qquad
( x \in B(\init), t \in \Rplus ),
\]
where $( G_h )_{h > 0}$ is a family of contractions in $B(\init \ot \khat )$
satisfying
$\hscale \big( G_h - I \big) \to F$ in norm as $h \to 0$.
Then
\begin{equation}
\label{eqn: inner}
\sup_{t \in [0,T]}
\big\| \big( j_t(x) - j_t^{(h)}(x) \big) \xi \big\| \to 0
\text{ as } h \to 0
\qquad
(x \in B(\init), \xi \in \init \ot \Fockk, T \in \Rplus).
\end{equation}
\end{cor}

\begin{proof}
Fix $T \in \Rplus$ and set $\Xsofth := \hXGh$.
Since $B(\init)$ is linearly spanned by its isometries,
it suffices to prove~\eqref{eqn: inner}
for $x$ isometric.
Accordingly,
let $x \in B(\init)$ be isometric.
By Theorem~\ref{thm: main},
\[
\sup_{t \in [0,T]}
\big\| \big( \Xsofth_t - X_t \big) \xi \big\| \to 0
\text{ as } h \to 0
\qquad
(\xi \in \init \ot \Fockk, T \in \Rplus).
\]
Since $X$ and $\Xsofth$ are contraction processes and thus both locally uniformly bounded,
it follows that
\[
\sup_{t \in [0,T]}
\big| \ip{\zeta}{\big( j_t(x) - j_t^{(h)}(x) \big) \eta} \big| \to  0
\text{ as } h \to 0
\qquad
(\zeta, \eta \in \init \ot \Fockk, T \in \Rplus).
\]
Now,
for each $t \in \Rplus$,
$j_t(x)$ is isometric and $j_t^{(h)}(x)$ is a contraction,
therefore~\eqref{eqn: inner} follows from Lemma~\ref{lemma: weak to strong}.
\end{proof}

%%%%%%%%%%%%%%%%%%%%%%%%%%%%%%%%%%%%%%%%%%%% SECTION

\section{Compositions}
\label{section: composition of quantum random walks}

In this section we show how, under the convergence scheme of Section~\ref{section: standard qrw approx},
pointwise products of QRWs converge to QS Trotter products of the limiting QS cocycles (\cite{LQST2}).
This specialises nicely to the case where the initial space $\init$ is a tensor product
and the two cocycles live on separate tensor components.

Recall the series-product notation~\eqref{eqn: QIP}.

\begin{thm}
\label{propn: composition}
Let $F_i$, $G_i(h) \in B(\init\ot \khat)$
\tu{(}$h > 0$\tu{)}, for $i = 1$, $2$,
let $c$, $d \in \noise$ and
suppose that
\begin{equation}
\label{eqn: half cgce}
E^{\chat} \big[ \hscale ( G_1 (h) - I ) - F_1 \big] \to 0
\ \text{ and } \
\big[ \hscale (G_2(h) - I) - F_2 \big] E_{\dhat} \to 0
\text{ as } h \to 0.
\end{equation}
Then
\[
E^{\chat} \big[ \hscale ( G_1(h)G_2(h)  - I ) -  F_1 \seriesprod F_2 \big] E_{\dhat} \to 0
\text{ as } h \to 0.
\]
Also,
if
$\hscale ( G_i (h) - I ) \to F_i$ \tu{(}in norm\tu{)} for $i = 1,2$,
then
$ \hscale ( G_1(h)G_2(h)  - I ) \to  F_1 \seriesprod F_2$.
\end{thm}

\begin{proof}
Let $h > 0$, and set
\[
F_1(h) := \hscale \big( G_1(h) - I \big), \
F_2(h) := \hscale \big( G_2(h) - I \big) \ \text{ and } \
F(h) := \hscale \big( G_1(h)G_2(h) - I \big).
\]
Then, from the identity
\[
G_1(h)G_2(h) - I =
 ( G_1(h) - I) + ( G_2(h) - I) +  ( G_1(h) - I ) ( G_2(h) - I ),
\]
we see that
\begin{align*}
F(h)
&
=
F_1(h) + F_2(h) +
F_1(h) \big( h \Delta^\perp + \Delta \big) F_2(h)
\\
&
=
F_1(h) + F_2(h) + F_1(h) \QIP F_2(h) +
h \, F_1(h) \Delta^\perp  F_2(h),
\end{align*}
from which both conclusions follow.
\end{proof}

\begin{rems}
(i)
Given an elementary QS operator cocycle $X$,
a series decomposition of its stochastic generator
\[
F = F_1 \seriesprod \cdots \seriesprod F_n,
\]
and families
 $\big( G_i(h) \big)_{h>0}$ in $B(\init\ot \khat)$ ($i = 1, \cdots , n$)
satisfying
\[
\hscale ( G_i (h) - I ) \to F_i
\qquad
(i = 1, \cdots , n),
\]
Theorems~\ref{propn: composition} and~\ref{thm: main} give that
$X$ is the limit of the embedded QRWs
$( X^{h,G_h } )_{h>0}$,
where $G_h := G_1(h) \cdots G_n(h)$.
This fact is exploited in the proof of Proposition~\ref{propn: impl},
and in Remark (ii) following Theorem~\ref{thm: Ghn}.

(ii)
If $F_i^* \seriesprod F_i \les 2 \beta_i \QIP^\perp$,
where $\beta_i \in \Real$ (respectively $F_i^* \seriesprod F_i =0$ or $F_i \seriesprod F_i^* =0$),
for $i = 1,2$,
then the QS cocycle
$X^{F_1 \seriesprod F_2}$ may be realised as a QS Trotter product of the
quasicontractive (respectively isometric or coisometric) QS cocycles $X^{F_1}$ and $X^{F_2}$
(\cite{LQST2}, Theorem 3.4).
\end{rems}

The following observation,
in which the Riesz--Nagy symbol $\smile$ denotes `commutes with' (\cite{RiN}),
is relevant here.

\begin{propn}[\cite{JuL}]
\label{propn: M and M}
Let $X^1$ and $X^2$ be quasicontractive QS cocycles on $\init$ with noise dimension space $\noise$.
Suppose that $X^1$ and $X^2$ commute on $\init$,
meaning that
\[
E^{\zeta_1} X^1_s E_{\eta_1}
\smile
E^{\zeta_2} X^2_t E_{\eta_2}
\qquad
( s,t \in \Rplus, \zeta_1, \eta_1, \zeta_2, \eta_2 \in \Fockk).
\]
Then the QS process $X^1X^2 := ( X^1_t X^2_t )_{t \ges 0}$
is also a quasicontractive QS cocycle.
Moreover,
if the QS cocycles $X^1$ and $X^2$ are both elementary
then  $X^1X^2 = X^{ F_1 \seriesprod F_2 }$
where $F_1$ and $F_2$ are the stochastic generators of $X^1$ and $X^2$ respectively.
\end{propn}

%%%%%%%%%%%%%%%%%%%%%%%%%%%%%%%%%%%%%%%%%%%%%%%%%%%%%%%%%%%%%%%%%%

\begin{example}
\label{ADPexample}
Let $X^{(i)}$ be a quasicontractive QS cocycle on $\init_i$ with noise dimension space $\noise$,
for $i = 1,2$.
These ampliate to QS cocycles on $\init := \init_1 \ot \init_2$,
by setting
$ I_{1}:= I_{\init_1}$, $I_2 := I_{\init_2}$,
\[
X^2_t := I_{1} \ot X^{(2)}_t
\ \text{ and } \
X^1_t :=I_{2} \flip X^{(1)}_t,
\]
where $B(\init_1) \otol B(\init_2\ot \Fockk)$ is identified with $B(\init \ot \Fockk)$
and the notation $\! \flip \!$ incorporates the tensor flip  from
$B(\init_2) \otol B(\init_1 \ot \Fockk)$ to $B(\init \ot \Fockk)$.
Since the $\Fockk$-slices of $X^1_s$ and $X^2_t$ belong to
$B(\init_1) \ot I_{2}$ and $I_{1} \ot B(\init_2)$ respectively,
the cocycles manifestly commute on $\init$.
Therefore,
by Proposition~\ref{propn: M and M},
the product
 $X^1 X^2$ is a quasicontractive QS cocycle.
Suppose now that,
 for $i = 1,2$,
$X^{(i)}$ is elementary with stochastic generator $F_{(i)}$,
and the family $( G^{(i)}_h )_{h>0}$ satisfies
$\hscale ( G^{(i)}_h - I_i ) \to F_{(i)}$,
and set
\[
F_2 := I_{1} \ot F_{(2)},
\
F_1 := I_{2} \flip F_{(1)},
\
G_2(h) := I_{1} \ot G^{(2)}_h
\ \text{ and } \
G_1(h) := I_{2} \flip G^{(1)}_h ,
\]
in which the tilde now incorporates the tensor flip
from
$B(\init_2) \otol B(\init_1 \ot \khat)$ to $B(\init \ot \khat)$.
Then
\[
X^1 X^2 = X^{F_1 \seriesprod F_2}
\ \text{ and } \
\hscale( G_i(h) - I_{\init \ot \khat} ) \to F_i
\quad
(i=1,2).
\]

In terms of the block matrix decompositions
$F_{(i)} = \left[\begin{smallmatrix} K_i & M_i \\ L_i & C_i - I \end{smallmatrix}\right]$
($i = 1,2$),
\[
F_1 \seriesprod F_2 =
 \begin{bmatrix}
K_1 \ot I_2 + I_1 \ot K_2 + (I_2 \flip M_1)(I_1 \ot L_2)
&
&
(I_2 \flip M_1 )(I_1 \ot C_2) + I_1 \ot M_2
\\
I_2 \flip L_1 +  (I_2 \flip C_1) (I_1 \ot L_2)
&
&
(I_2 \flip C_1) (I_1 \ot C_2) - I
\end{bmatrix}.
\]
In the case of one-dimensional noise this simplifies to
\[
\begin{bmatrix}
K_1 \ot I_2 + I_1 \ot K_2 + M_1 \ot L_2
&
 M_1 \ot C_2+ I_1 \ot M_2
\\
L_1 \ot I_2 +  C_1 \ot L_2
&
C_1 \ot C_2 - I
\end{bmatrix},
\]
whereas the quantum random walk generator satisfies
\begin{equation*}
G_1(h) G_2(h) =
\begin{bmatrix}
I + h K_h & \sqrt{h} M_h \\ \sqrt{h} L_h & C_h \end{bmatrix}
+ h \, O(h)
\end{equation*}
where,
writing
$G_i(h)$ in the form
$\left[\begin{smallmatrix}
I + h K_i(h) & \sqrt{h} M_i(h) \\ \sqrt{h} L_i(h) & C_i(h)
 \end{smallmatrix}\right]$
for
$i = 1,2$,
$\left[\begin{smallmatrix} K_h &  M_h \\  L_h & C_h \end{smallmatrix}\right]$
equals
\begin{equation*}
\begin{bmatrix}
 K_1(h) \ot I_2 + I_1 \ot K_2(h) + M_1(h) \ot L_2(h)
&
&
 M_1(h) \ot C_2(h)+ I_1 \ot M_2(h)
\\
 L_1(h) \ot I_2 +  C_1(h) \ot L_2(h)
&
&
C_1(h) \ot C_2(h)
\end{bmatrix}
\end{equation*}
and
\[
O(h) =
\begin{bmatrix}
h \,  \big( K_1(h) \ot K_2(h) \big)
&
\sqrt{h} \big( K_1(h) \ot M_2(h) \big)
\\
\sqrt{h} \big( L_1(h) \ot K_2(h) \big)
&
L_1(h) \ot M_2(h)
\end{bmatrix}.
\]
\end{example}

\begin{rem}
Example~\ref{ADPexample} specialises to the entanglement of bipartite systems
under repeated interactions, as considered in~\cite{ADP}.
This is fully elaborated upon in
Section~\ref{section: repeated quantum interactions} below.
\end{rem}

%%%%%%%%%%%%%%%%%%%%%%%%%%%%%%%%%%%%%%%%%%%%%%%%%%%%%%%%%%%%%%%%%%
% END OF SECTION

%%%%%%%%%%%%%%%%%%%%%%%%%%%%%%%%%%%%%%%%%%%%%%%%%

%%%%%%%%%%%%%%%%%%%%%%%%%%%%%%%%%%%     SECTION

\section{Holevo transform}
\label{section: transformation}

In this section we
consider a nonlinear transformation of block operator matrices in
$B(\init \ot \khat) = B(\init \op (\init \ot \noise ))$
intimately
related to the convergence to QS cocycles for two classes of QRW.
Its origin lies in Holevo's approach to realising QS cocycles,
and more general solutions of QS differential equations,
as time-ordered exponentials
([Ho$_{1,2}$]; see [Ho$_{3,4}$]).
Since time-ordered exponentials may be seen as explicit continuous counterparts to QRWs,
the appearance of this transform is, with hindsight, not unexpected.
On the one hand the transform is key to the realisation of QS cocycles
as limits of scaled QRWs described in Section~\ref{section: realisation}.
On the other hand it delivers convergence of QRWs to unitary QS cocycles
in the repeated quantum interaction model,
as shown in Section~\ref{section: repeated quantum interactions}.
In particular,
part (a) of Proposition~\ref{propn: Q skew} and part (b) of Theorem~\ref{thm: Ghn}
show clearly the origin of what, in~\cite{AtP},
is referred to as the `surprising term' in the `effective Hamiltonian'
arising in the limit of repeated quantum interactions
(see Theorem~\ref{thm: rqi} below).

Set
\[
\Al:=
 \big\{
 T \in
 B( \init \op ( \init \ot \noise ) \op \init ):
 P_{ \init \op ( \init \ot \noise ) \op \{ 0 \} } T
 = T =
 T P_{ \{ 0 \} \op ( \init \ot \noise ) \op \init }
 \big\},
 \]
 thus $\Al$ consists of the elements having block matrix form
 $\left[ \begin{smallmatrix} 0 & * & * \\ 0 & * & * \\ 0 & 0 & 0 \end{smallmatrix} \right]$.
The unital Banach algebra
$B( \init \op ( \init \ot \noise ) \op \init )$
has an involution given by
\[
T \mapsto T^\star := \Xi T^* \Xi,
\ \text{ where } \
\Xi :=
\begin{bmatrix}
0 & 0 & I_\init \\ 0 & I_{\init\ot\noise} & 0 \\ I_\init & 0 & 0
\end{bmatrix}
\]
with respect to which $\Al$ is a closed *-subalgebra.
 The prescription
 \[
 \begin{bmatrix}
 A & C \\ B & D
 \end{bmatrix}
 \mapsto
 \begin{bmatrix}
 0 & C & A \\ 0 & D & B \\ 0 & 0 & 0
 \end{bmatrix}
 \]
 defines an involutive linear homeomorphism
$\tau: ( B( \init \ot \khat ), * ) \to (\Al, \star)$,
and
the
involutive continuous (nonlinear) injection
 \[
 B( \init \op ( \init \ot \noise ) \op \init )
 \to
 B( \init \op ( \init \ot \noise ) \op \init ),
 \quad
 T \mapsto e^T - I,
 \]
 restricts to a map
 $\eta: \Al \to \Al$.
 The
 \emph{Holevo transform} is
 the map
  \begin{equation*}
 %\label{eqn: Htransf}
 \tau^{-1} \circ \eta \circ \tau:
 B( \init \ot \khat ) \to  B( \init \ot \khat ),
 \quad
 Q \mapsto F[Q] := \tau^{-1} \big( e^{ \tau(Q) } - I \big).
 \end{equation*}
 If we set
$\mathcal{S} :=  I_{ \init \op ( \init \ot \noise ) \op \init } + \Al$
and let
$\wt{\tau}$ denote the map $Q \mapsto \tau(Q) + I$,
then,
recalling the remark below~\eqref{eqn: QIP},
$( \mathcal{S}, \cdot , \star )$ is a *-monoid,
$\wt{\tau}$ is a *-monoid isomorphism
from
$( B( \init \ot \khat ), \seriesprod, *)$ to $( \mathcal{S}, \cdot , \star )$
(\cite{Belavkin},~\cite[Proposition 1.5]{LQST2}),
and
$F[Q] = \wt{\tau}^{-1} \big( e^{ \wt{\tau}(Q) - I } \big)$.
From these representations it is readily verified that
$F[Q]^* = F[ Q^*]$ and
\begin{equation}
\label{eqn: Qstruct}
F[Q]^* \seriesprod F[Q] = 0
\iff
e^{\tau(Q^*)} = e^{\tau(-Q)}
\iff
F[Q] \seriesprod F[Q]^* = 0,
\end{equation}
so that $F[Q]$ satisfies the unitary structure relations (see Theorem~\ref{thm: cocycles})
if $Q$ is skewadjoint.
%and
%we see that~\eqref{eqn: FQu} holds if $Q$ is skewadjoint.

In order to examine the Holevo transform in more detail, and to
show its use, we need to introduce some functions and
give various relations they enjoy.
Thus,
let
$e_0, e_1, e_2$ and $e$ be the entire functions
whose values at $z \neq 0$ are given respectively by
\begin{equation}
\label{eqn: entire}
e^z, \quad
\frac{e^z - 1}{z}, \quad
\frac{e^z - 1 - z}{z^2}
 \ \text{ and } \
\frac{ e^z - e^{-z} - 2z }{2 z^2} = \frac{ \sinh z - z}{z^2},
\end{equation}
and, for $n \in \Nat$, let $p_n$ denote the polynomial whose value at $z \neq 0$ is given by
\[
\frac{ (1 + z/n )^n - 1 - z }{z^2}.
\]
Thus
$e_0(0) = e_1(0) =  1$, $e_2(0) = 1/2$, $e(0) = 0$,
$p_n \to e_2$ uniformly on bounded subsets of $\Comp$
and the following identities hold,
for all $z \in \Comp$ and $n \in \Nat$:
\begin{subequations}
\label{eqn: identities}
\begin{align}
& \label{eqn: pn}
z^n - 1 =
n(z-1) + n(z-1)p_n\big( n(z-1) \big) n(z-1),
\\
& \label{eqn: e20}
1 + z e_2(z) = e_1(z)
\ \text{ and } \
z + z^2 e_2(z) = e_0(z) - 1,
\\
& \label{eqn: e02}
n \big( e_0(z/n) - 1 \big) - z = e_2(z/n)  z^2/n,
\\
& \label{eqn: e1}
e_1(-z) e_0(z) = e_1(z)
\ \text{ and } \
  \tfrac{1}{2} e_1(-z) e_1(z) + e(z) =  e_2(z).
\end{align}
\end{subequations}
\noindent
In terms of these,
the Holevo transform
is given by
\begin{equation}
\label{eqn: transf}
Q
=
 \begin{bmatrix} A & C \\ B & D \end{bmatrix}
 \
\mapsto
\
F[Q]
:=
\begin{bmatrix}
A + C e_2(D) B & C e_1(D) \\ e_1(D) B & e_0(D) - I
\end{bmatrix}.
\end{equation}
Thus, by~\eqref{eqn: e20},
\begin{equation}
\label{eqn: F[Q]}
F[Q] - Q =
\begin{bmatrix} C \\ D \end{bmatrix}
e_2(D)
\begin{bmatrix} B & D \end{bmatrix}
=
Q \begin{bmatrix} 0 & 0 \\ 0 & e_2(D) \end{bmatrix} Q.
\end{equation}
We next look at how parameterisations of $Q$
are reflected in parameterisations of $F[Q]$.
For operators
$A \in B(\init)$,
$B \in B(\init; \init \ot \noise)$,
$C \in B(\init \ot \noise; \init)$ and
$D \in B(\init \ot \noise)$,
set
\begin{equation}
\label{eqn: Qlft}
Q_{A, B, D} :=
\begin{bmatrix}
A  & -B^* \\ B & D
\end{bmatrix}.
\end{equation}
In terms of the notation~\eqref{eqn: Flft},
note
the relations
\begin{equation*}
F[ Q_{A, 0, D} ]
=
F_{A, 0, e^D}
\ \text{ and } \
F_{Z, L, I}
 =
F[ Q_{Z, L, 0} ].
\end{equation*}
Below we extend the scope of this correspondence
(for skewadjoint $D$).
To this end let
$e_a, e_b: [0,2\pi[ \to \Comp$ denote the continuous functions
whose values at $t \in ]0, 2\pi[$ are given respectively by
\[
\frac{i}{2} \frac{ \sin t - t }{ \cos t - 1 }
\ \text{ and } \
\frac{it}{ e^{it} - 1},
\]
and note the following easily verified identities,
for
$t \in [0, 2\pi[$,
\begin{equation}
\label{eqn: e one a}
\overline{ e_a(t) } = - e_a(t),
\
e_1(it) e_b(t) = 1,
\
\overline{ e_b(t) } e_1(it) =  e_0(it),
\
%\text{ and } \
| e_b(t) |^2 e(it) = e_a(t).
\end{equation}

%\newpage
\begin{propn}
\label{propn: Q skew}
\mbox{}
\begin{alist}
\item
Let $A \in B(\init)$,
$B \in B(\init; \init \ot \noise)$ and
$D \in B(\init \ot \noise)$,
with $D$ skewadjoint.
Then
\[
F[Q_{A, B, D}] = F_{Z, L, W}
\]
where
\[
W = e_0(D),
\quad
L = e_1(D) B
\ \text{ and } \
Z = A -  B^* e(D) B
\]
\tu{(}so that
$W$ is unitary
and
$Z - A$ is skewadjoint\tu{)}.
In particular,
if $Q \in B(\init \ot \khat)$ is skewadjoint then
\begin{equation}
\label{eqn: FQu}
F[Q]^* \seriesprod  F[Q] = 0 =  F[Q] \seriesprod  F[Q]^*
\end{equation}
and so the QS cocycle $X^{F[Q]}$ is unitary.

\item
Conversely,
let $Z \in B(\init)$,
$L \in B(\init; \init \ot \noise)$ and
$W \in B(\init \ot \noise)$,
with
$W$ unitary
and satisfying
$\Spec W \subset \big\{ e^{it}: t \in [0, 2 \pi [ \big\}$.
Then
\[
F_{Z, L, W} = F[Q_{A, B, D}],
\]
where
\[
A = Z + L^* e_a(R) L,
\
B = e_b(R) L
\ \text{ and } \
D = iR,
\]
for the unique selfadjoint opertator
$R \in B(\init \ot \noise)$
satisfying
$e^{iR} = W$ and
$\Spec R \subset [0, 2 \pi [$
\tu{(}so that
$A - Z$ is skewadjoint\tu{)}.
In particular,
if $Z$ is skewadjoint then so is $Q_{A, B, D}$,
and so
the operator
$e^{Q_{A, B, D}}$ is unitary.
\end{alist}
\end{propn}

\begin{proof}
(a)
By the skewadjointness of $D$, and oddness of the function $e$,
$W$ is unitary and $Z$ is skewadjoint.
Moreover, by~\eqref{eqn: e1},
\begin{equation*}
- L^* W = - B^* e_1(D)^* e_0(D) = - B^* e_1(-D) e_0(D) = - B^* e_1(D),
\end{equation*}
and
\begin{align*}
&
- \tfrac{1}{2} L^* L - B^* e(D) B =
-B^* \big[ \tfrac{1}{2} e_1 (D)^* e_1(D) + e(D) \big] B
\\
&
\qquad \qquad \qquad \qquad
 \quad
=
-B^* \big[ \tfrac{1}{2} e_1 (- D)  e_1(D) + e(D) \big] B
=
-B^* e_2(D) B,
\end{align*}
so $F[Q_{A, B, D}]$ has the claimed form.

Now suppose that
$Q \in B(\init \ot \khat)$ is skewadjoint.
Then $Q$ is of the form
$Q_{A, B, D}$
with $A$ and $D$ skewadjoint,
so
$F[Q]$ is of the form
$F_{Z, L, W}$  with $Z$ skewadjoint and $W$ unitary,
and thus~\eqref{eqn: FQu} holds by property (i) above Theorem~\ref{thm: cocycles}
(confirming the remark which follows observation in~\eqref{eqn: Qstruct}).

\noindent
(b)
Let $R$ be as specified.
Using the identities~\eqref{eqn: e one a}, and Part (a),
we see that
$A - Z = L^* e_a(R) L$
is skewadjoint and, since $D = iR$ is skewadjoint,
$F[ Q_{A, B, D} ] = F_{\wt{Z}, \wt{L}, W}$
where
\begin{align*}
&
\wt{L} = e_1(iR) e_b(R) L = L,
\text{ and }
\\
&
\wt{Z} =
A - L^* e_b(R)^* e(iR) e_b(R) L
=
A - L^* e_a(R) L = Z.
\end{align*}
Thus
$F[ Q_{A, B, D} ] = F_{Z, L, W}$.

Now suppose that $Z$ is skewadjoint.
Then
$A = Z + L^* e_a(R) L$ is skewadjoint and so $Q_{A, B, D}$ is too.
\end{proof}

We may now see precisely how the Holevo transform relates to
the convergence of scaled quantum random walks.
\begin{thm}
\label{thm: Ghn}
Let $Q  \in B(\init \ot \khat)$.

\begin{alist}

\item
Suppose that $( P(h,n) )_{h > 0, n \in \Nat}$
is a family in $ B(\init \ot \khat)$ satisfying
\[
n \mathit{s}_h \big( P(h,n) - I \big) \to Q
\text{ as }
h \to 0
\text{ and }
n \to \infty.
\]
Then
\begin{equation}
\label{eqn: Ghnn}
\mathit{s}_h \big( P(h, n)^n - I \big) \to F[Q]
\text{ as }
h \to 0
\text{ and }
n \to \infty,
\end{equation}
and so
if
$Q$ is skewadjoint and
each $P(h, n)$
is contractive
then,
for all $\xi \in \init \ot \Fockk, T \in \Rplus$,
\begin{align*}
\sup_{t \in [0,T]}
\Big(
\big\|
\big(
\hXPhn_t - X^{F[Q]}_t
 \big)
\xi
 \big\|
+
\big\|
\big(
\hXPhn_t
&
- X^{F[Q]}_t
 \big)^*
\xi
 \big\|
\Big)
%\\
%&
\to 0
\end{align*}
as
$h \to 0$
and
$n \to \infty$.

\item
Suppose that $( Q_h )_{h>0}$
is a family in $ B(\init \ot \khat)$ satisfying
\[
\mathit{s}_h ( Q_h ) \to Q
\text{ as }
h \to 0.
\]
Then
\begin{equation}
\label{eqn: eEh}
\mathit{s}_h \big( e^{Q_h} - I \big) \to F[Q]
\text{ as }
h \to 0,
\end{equation}
and so
if
$Q$ is skewadjoint
and
each $Q_h$ is dissipative
then,
for all $\xi \in \init \ot \Fockk, T \in \Rplus$,
\[
\sup_{t \in [0,T]}
\Big(
\big\|
\big(
X^{h,e^{Q_h}} _t - X^{F[Q]}_t
 \big)
\xi
\big\|
+
\big\|
\big(
X^{h,e^{Q_h}} _t - X^{F[Q]}_t
 \big)^*
\xi
\big\|
\Big)
\to 0
\text{ as }
h \to 0.
\]
\end{alist}
\end{thm}

\begin{proof}
Let
$\left[ \begin{smallmatrix} A & C \\ B & D \end{smallmatrix}\right]$
be the block matrix form of $Q$ and,
for $h > 0$, set
\[
\Delta_h := I_\init \ot ( \mathcal{S}_h^\noise )^{-1}  =
\begin{bmatrix}
\sqrt{h} I_\init & 0 \\ 0 & I_{\init \ot \noise}
\end{bmatrix}.
\]
In both (a) and (b) the final claims follow from
Proposition~\ref{propn: Q skew} and
Theorem~\ref{thm: main}
along with
Remark (i) following it.

(a)
For $h > 0$ and $n \in \Nat$ define the operator
\[
Q_{h,n} :=
n \mathit{s}_h \big( P(h,n) - I \big) \in B(\init \ot \khat).
\]
Then, invoking the identity~\eqref{eqn: pn},
we see that
\[
\mathit{s}_h \big( P(h,n)^n - I \big) =
Q_{h,n} + Q_{h,n} \Delta_h p_n \big( \Delta_h Q_{h,n} \Delta_h \big) \Delta_h Q_{h,n}.
\]
Now, as $h \to 0$ and $n \to \infty$,
$Q_{h,n} \to Q$ and $\Delta_h \to \Delta$, so

\[
\Delta_h Q_{h,n} \to
\begin{bmatrix}
0 & 0 \\ B & D
\end{bmatrix},
\ \
Q_{h,n} \Delta_h
\to
\begin{bmatrix}
0 & C \\ 0 & D
\end{bmatrix}
\text{ and }
\Delta_h Q_{h,n} \Delta_h \to
\begin{bmatrix}
0 & 0 \\ 0 & D
\end{bmatrix}.
\]
Thus,
since $p_n \to e_2$ uniformly on compact subsets of $\Comp$,
identity~\eqref{eqn: F[Q]}
implies that
\begin{align*}
\mathit{s}_h \big( P(h,n)^n - I \big)
&
\to
Q + Q
\begin{bmatrix}
0 & 0 \\ 0 & e_2(D)
\end{bmatrix}
Q
=
F[Q]
\ \text{ as }  h \to 0 \text{ and } n \to \infty.
\end{align*}

(b)
It follows from the identity~\eqref{eqn: e02} that
\begin{align*}
n \mathit{s}_h \big( e^{ Q_h/n } - I \big)
- \mathit{s}_h \big( Q_h \big)
&
=
\mathit{s}_h \big( Q_h e_2 \big(  Q_h / n \big) Q_h \big) / n
\\
&
=
\mathit{s}_h \big( Q_h \big) \Delta_h
e_2 \big(   \Delta_h \mathit{s}_h ( Q_h )  \Delta_h  / n \big)
 \Delta_h \mathit{s}_h \big( Q_h \big) / n.
\end{align*}
Now
 $e_2$ is continuous at $0$
 and, as $h \to 0$,
$\mathit{s}_h \big( Q_h \big) \to Q$  and $\Delta_h \to \Delta$,
therefore
\[
n \mathit{s}_h \big( e^{ Q_h / n } - I \big) \to Q
\ \text{ as } h \to 0 \text{ and } n \to \infty,
\]
and so~\eqref{eqn: eEh} holds by (a).
\end{proof}

\begin{rem}
Part (b)
may be compared with (22) in~\cite{Gough}
and
Theorem~19 of~\cite{AtP}.
\end{rem}

Appealing to Theorem~\ref{propn: composition},
we see that Theorem~\ref{thm: Ghn}
has the following consequence.

\begin{cor}
\label{cor: Qoneplustwo}
Suppose that, for $i=1,2$,
$\mathit{s}_h (Q_i(h)) \to Q_i$ as $h \to 0$,
for operators
$Q_i, Q_i(h) \in B(\init \ot \khat)$ $(h>0)$.
Then
\[
\mathit{s}_h \big( e^{ Q_1(h) + Q_2(h) } - I \big) \to F[Q_1 + Q_2]
\ \text{ and } \
\mathit{s}_h \big( e^{ Q_1(h) } e^{ Q_2(h) } - I \big) \to
F[Q_1] \seriesprod F[Q_2].
\]
\end{cor}

%%%%%%%%%%%%%%%%%%%%%%%%%%%%%%%%%%%     SECTION

\section{Realisations}
\label{section: realisation}

In this section we
give a variety of ways of implementing
the approximation schemes of Theorems~\ref{thm: main} and~\ref{thm: Ghn},
with the assistance of Theorem~\ref{propn: composition}.
In particular, we show that each kind of QS cocycle (quasicontractive, isometric, coisometric or unitary)
may be obtained as a limit of QRWs of the same kind.

Corresponding to the assemblies~\eqref{eqn: Flft},
set
\[
V_{Z, L, W} :=
(e^{Z} \oplus I) \, V_L^\lft \, (I \oplus W),
\]
where
\begin{align*}
&
V_L :=
\begin{bmatrix}
(I + L^* L )^{-1/2} & - L^* (I + L L^*  )^{-1/2}
\\
 L  (I + L^* L )^{-1/2} & (I + L L^*  )^{-1/2}
\end{bmatrix},
%\ \text{ and } \
%\\
%&
%V^\rht_M :=
%\begin{bmatrix}
%(I + M M^*  )^{-1/2} & M (I + M^* M  )^{-1/2}
%\\
% -M^* (I + M M^* )^{-1/2} & (I + M^* M )^{-1/2}
%\end{bmatrix},
\end{align*}
for
$Z \in B( \init )$,
$L \in B(\init; \init \ot \noise )$, $M \in B(\init \ot \noise; \init )$
and
$W \in B(\init \ot \noise)$.
Note that
$V_L$ is unitary.
%\begin{equation}
%\label{eqn: Vrstar}
% V^{\rht}_{Z, M, W}  = ( V_{Z^*, M^*, W^*} )^*
% \ \text{ and } \
% V^\rht_M := V_{-M^*}.
%\end{equation}

We use the following abbreviation below:
\begin{equation}
\label{eqn: pos part}
T_+ := ( \re T )_+
\ \text{ for } \
T \in B( \Hil ).
\end{equation}

\begin{lemma}
\label{lemma: Z}
Let $Z \in B(\Hil)$ for a Hilbert space $\Hil$. Then
\[
\norm{ e^Z } \les e^{ \norm{Z_+} }.
\]
\end{lemma}
\begin{proof}
The operator $Y:= Z - Z_+$ is dissipative
and therefore
\begin{equation*}
 \norm{ ( e^{ Y/n} e^{ Z_+/n} )^n}
\les
( \norm{ e^{ Y/n} }
\norm{ e^{ Z_+/n} } )^n
\les
\norm{ e^{ Z_+/n} }^n
\les
(e^{\norm{ Z_+}/n } )^n
=
e^{ \norm{Z_+} }
\quad
(n \in \Nat),
\end{equation*}
so the result follows from
the Lie--Trotter product formula
(\cite{ReS}).
\end{proof}

%%%%%%%%%%%%%%%%%%%%%%%%%%%%%%%%%%%      COROLLARY
\begin{propn}
\label{propn: impl}
Let
$T \in B( \init \ot \khat)$,
$Z \in B(\init)$,
$L \in B(\init; \init \ot \noise)$,
$M \in B(\init \ot \noise; \init)$
and
$W, R \in B(\init \ot \noise)$,
where $W$ is contractive and $R$ is selfadjoint with $\Spec R \subset [0, 2\pi[$.
Set
\begin{equation*}
T^0_0 := E^{\wh{0}} T E_{\wh{0}}
\ \text{ and } \
Q_h :=
Q_{h( Z + L^* e_a(R) L), \sqrt{h} \,e_b(R) L, \, iR}
\qquad
(h > 0),
\end{equation*}
in the notation~\eqref{eqn: Qlft}.
Then,
for $h > 0$ and $t \in \Rplus$,
\begin{equation}
\label{eqn: Tzero}
\norm{ e^{h T} }^{ \lfloor t/h \rfloor }
\les
e^{ t \norm{ T_+ } },
\ \
\norm{ e^{ Q_h} }^{ \lfloor t/h \rfloor }
\les
e^{ t \norm{ Z_+ } }
\ \text{ and } \
\norm{ V_{h Z, \sqrt{h} L, W} }^{ \lfloor t/h \rfloor }
\les
e^{ t \norm{ Z_+ } }.
%\ \text{ and } \
%\norm{ V^\rht_{h Z, \sqrt{h} M, W} }^{ \lfloor t/h \rfloor }
%\les
%e^{ t \norm{ Z_+ } }.
\end{equation}
Moreover, as $h \to 0$,
\begin{align*}
&
\hscale \big( e^{ Q_h}  - I \big) \to F_{Z, L, e^{iR}}
\ \text{ and } \
%\\
%&
\hscale \big( e^{h T} V_{h Z, \sqrt{h} L, W}  - I \big) \to F_{T^0_0 + Z, L, W}
%\ \text{ and } \
%\hscale \big( e^{h T} V^\rht_{h Z, \sqrt{h} M, W}  - I \big) \to F^\rht_{T^0_0 + Z, M, W}.
\end{align*}
\end{propn}

\begin{proof}
Since $( Q_h )_+ =
\left[ \begin{smallmatrix} h Z_+ & 0 \\ 0 & 0 \end{smallmatrix}\right]$
and
$h \lfloor t/h \rfloor \les t$
($h > 0$, $t \in \Rplus$),
the inequalities~\eqref{eqn: Tzero}
follow from Lemma~\ref{lemma: Z}
and the unitarity of
$V_{ L }$,
% and $V^\rht_{  M }$.

As $h \to 0$
\[
\hscale ( V_{ \sqrt{h} L } - I ) =
\begin{bmatrix}
h^{-1} \big[ (I + h L^* L )^{-1/2} - I \big]  & - L^* (I + h L L^*  )^{-1/2}
\\
 L  (I + h L^* L )^{-1/2} & (I + h L L^*  )^{-1/2}  - I
\end{bmatrix}
\to
F_{0, L, I}
\]
and
\[
\hscale
( e^{h T} - I )
=
\hscale
(h T ) + O(h)
\to
 F_{T^0_0, 0, I}.
\]
Moreover, for all $h > 0$,
\[
\hscale ( I \oplus W - I ) = F_{0, 0, W}.
\]
Therefore,
by Theorem~\ref{propn: composition}
and identity~\eqref{eqn: thus c}
\begin{align*}
\hscale \big( e^{h T} V_{h Z, \sqrt{h} L, W}  - I \big)
&
=
\hscale \big( e^{h T} e^{ h(Z \oplus 0) } V_{ \sqrt{h} L } (I \oplus W)  - I \big)
\\
&
\to
F_{T^0_0, 0, I } \seriesprod
F_{Z, 0, I } \seriesprod
F_{0 , L, I} \seriesprod
F_{0, 0, W}
=
F_{T^0_0 + Z, L, W}.
\end{align*}
%Similarly,
%by virtue of the adjoint relations~\eqref{eqn: Vrstar} and~\eqref{eqn: Frstar},
%\begin{equation*}
%\hscale \big( e^{h T} V^\rht_{h Z, \sqrt{h} M, W}  - I \big)
%\to
%F^\rht_{T^0_0 + Z,  M, W}.
%\end{equation*}
Finally,
since for all $h > 0$
\[
\hscale ( Q_h ) =
Q_{ Z + L^* e_a(R) L, e_b(R) L, iR},
\]
Theorem~\ref{thm: Ghn} (b) and Proposition~\ref{propn: Q skew} (b)
imply that
\[
\hscale ( e^{Q_h} - I ) \to F_{Z,  L,  e^{iR}}.
\qedhere
\]
\end{proof}

We next show that,
given an elementary QS cocycle $X$ which is
either
quasicontractive with exponential growth bound $\beta$,
isometric,
coisometric, or
unitary,
we may easily construct (from its stochastic generator)
QRWs which are respectively
quasicontractive with exponential growth bound $\beta$,
isometric,
coisometric, or
unitary,
and enjoy
locally uniformly strong convergence to $X$.
We need the following lemma,
which expresses a functorial property
common to the
generation of QS cocycles and that of embedded QRWs.

\begin{lemma}
\label{lemma: JGJ}
Let $J \in B( \noise; \Kil)$
be an isometry into a Hilbert space $\Kil$,
and set
\begin{align*}
&
J_\noise :=
I_\init \ot ( I_\Comp \oplus J ) \in B( \init \ot \khat; \init \ot \Khat )
\text{ and }
\\
&
J_\Fock :=
I_\init \ot \Gamma ( I_{L^2(\Rplus)} \ot  J ) \in B( \init \ot \Fock^\noise; \init \ot \Fock^{\Kil} ).
\end{align*}
Let $F, G \in B( \init \ot \Khat )$.
Then
\[
J_\Fock^* X^F_t J_\Fock =
X^{J_\noise^* F J_\noise}_t
\text{ and }
J_\Fock^* X^{h,G}_t J_\Fock =
X^{h,J_\noise^* G J_\noise}_t
\qquad
( t \in \Rplus, h > 0 ).
\]
\end{lemma}

\begin{proof}
Let us adopt the following notation for
$c \in \noise$,
$f \in L^2(\Rplus; \noise)$ and $Q \in B( \init \ot \Khat )$\tu{:}
\[
Jf := J f( \cdot ),
\ \text{ and } \
Q' :=   J_\noise^* Q J_\noise.
\]
Thus
\[
E^{\chat} Q' E_{\dhat} = E^{\wh{Jc}} Q E_{\wh{Jd}}
\qquad
(c, d \in \noise).
\]
Let
$F, G \in B( \init \ot \Khat )$ and $h > 0$.

(a)
Set $Y := ( J_\Fock^* X^F_t J_\Fock )_{t \ges 0}$.
The first identity follows from uniqueness for
weakly regular weak solutions of the QS differential equation
$\rd X_t = X_t  \ \! \rd \Lambda_{F'}(t)$  with $X_0 = I_{\init \ot \Fockk}$
(Theorem~\ref{thm: cocycles})
since,
for $u, v \in \init$, $f, g \in L^2(\Rplus; \noise)$ and $t \in \Rplus$,
\begin{align*}
\ip{ u \ve(f) }{ ( Y_t - I ) v \ve(g) }
&
=
\ip{ J_\Fock  u \ve(f) }{ ( X^F_t - I ) J_\Fock v \ve(g) }
&
\\
&
=
\ip{ u \ve(Jf) }{ ( X^F_t - I )  v \ve(Jg) }
\\
&
=
\int_0^t \rd s \,
\ip{ u \ve(Jf) }{  X^F_s E^{\wh{ Jf(s)}} F E_{\wh{ Jg(s)}} v \ve(Jg) }
\\
&
=
\int_0^t \rd s \,
\ip{ J_\Fock u \ve(f) }{  X^F_s E^{\wh{ f(s)}} F' E_{\wh{ g(s)}} J_\Fock  v \ve(g) }
\\
&
=
\int_0^t \rd s \,
\ip{  u \ve(f) }{  Y_s E^{\wh{ f(s)}} F' E_{\wh{ g(s)}}  v \ve(g) },
\end{align*}
so $Y = X^{F'}$.

(b)
Set $Y := ( J_\Fock^* X^{h,G_t} J_\Fock )_{t \ges 0}$
and let
 $f, g \in L^2(\Rplus; \noise)$ and $t \in \Rplus$.
Then
\[
 E^{\ve(f)} Y_t \, E_{\ve(g)}
=
\alpha
\prod_{
0 \les p < \lfloor t/h \rfloor
}^{\longrightarrow}
A_p
\ \text{ and } \
 E^{\ve(f)} \, X^{h, G'}_t \, E_{\ve(g)}
=
\alpha
\prod_{
0 \les p < \lfloor t/h \rfloor
}^{\longrightarrow}
B_p,
\]
where
\begin{align*}
&
A_p :=
E^{ \ve( Jf_{[hp, h(p+1)[})}
\,
X^{h,G_{hp, h(p+1)}}
\,
E_{ \ve( Jg_{[hp, h(p+1)[})},
\\
&
B_p :=
E^{ \ve( f_{[hp, h(p+1)[})}
\,
X^{h,G'}_{hp, h(p+1)}
\,
E_{ \ve( g_{[hp, h(p+1)[})},
\ \text{ and }
\\
&
\alpha :=
\ip{ \ve( f_{[ h j, \infty [} )}{ \ve( g_{[ h j, \infty [} )}
\ \text{ where } \
j = \lfloor t/h \rfloor.
\end{align*}
These products coincide since,
for $p \in \Zplus$ and $u, v \in \init$,
setting
$K = [hp, h(p+1)[$
and recalling the notation~\eqref{eqn: average},
\begin{align*}
\ip{ u  \ve( f_K ) }{ ( X^{h, G'}_{hp, h(p+1)} - I ) v  \ve( g_K ) }
&
=
\ip{ u \wh{ \sqrt{h} Jf[p,h]} }{ ( G - \Delta^\perp ) v \wh{ \sqrt{h} Jg[p,h]} }
\\
&
=
\ip{ u \wh{ \sqrt{h} f[p,h]} }{ (G' - {\Delta'}^\perp ) v \wh{ \sqrt{h} g[p,h]} }
\\
&
=
\ip{ u  \ve( Jf_K ) ) }{ ( X^{h,G}_{hp, h(p+1)} - I ) v  \ve( Jg_K ) )}
\end{align*}
(\emph{cf}. the proof of Lemma~\ref{lemma: evol to Euler}),
and
\[
\ip{ u \ve( Jf_K )}{ v \ve( Jg_K )}
=
\ip{u}{v} e^{\ip{f_K}{g_K}}
=
\ip{ u \ve( f_K )}{ v \ve( g_K )},
\]
so
$Y = X^{h, G'}$.
\end{proof}

We are now ready to fulfill the promise contained in the remark following Corollary~\ref{cor: inner}.

%%%%%%%%%%%%%%%%%%%%%%%%%%%%%%%%%%% thm
\begin{thm}
\label{thm: strong cgce}
Let $X$ be an elementary QS cocycle on $\init$
with noise dimension space $\noise$.
\begin{alist}
\item
Suppose that $X$ is quasicontractive with exponential growth bound $\beta$
and stochastic generator $F$.
Then there is a family $(G_h)_{h>0}$ in $B(\init\ot\khat)$
such that
\begin{subequations}
\begin{align}
&
\label{subeq: y}
\ \
\norm{ G_h }^{ \lfloor t/h \rfloor } \les e^{ t \beta }
\quad
(h > 0, t \in \Rplus),
\\
&
\label{subeq: z}
\ \
\mathit{s}_h ( G_h - I )
\to
F
\
\text{ as }
\
h \to 0,
\\
&
\label{subeq: a}
\sup_{t\in [0,T] }
\big\|
\big(
X_t - X^{h,G_h}_t
 \big)\xi
 \big\|
\to 0
\text{ as }
h \to 0
\quad
(\xi \in \init \ot \Fockk, T \in \Rplus)
\ \text{ and } \
\\
&
\label{subeq: b}
\sup_{t\in [0,T]}
\big\|
\big(
X_t - X^{h,G_h}_t
 \big)^*
\xi
 \big\|
\to 0
\text{ as }
h \to 0
\quad
(\xi \in \init \ot \Fockk, T \in \Rplus).
\end{align}
\end{subequations}

\item
Suppose that $X$ is
isometric, coisometric, or unitary.
Then
there is a family
$(G_h)_{h>0}$ in $B(\init\ot\khat)$ satisfying~\eqref{subeq: z}
such that, respectively,
each $G_h$ is
isometric and~\eqref{subeq: a} holds,
each $G_h$ is
coisometric and~\eqref{subeq: b} holds,
or
each $G_h$ is
unitary and~\eqref{subeq: a}, or equivalently~\eqref{subeq: b}, holds.
\end{alist}
\end{thm}
\begin{proof}
The proof is in three parts.

(1)
Suppose first that $X$ is isometric.
Then,
by Theorem~\ref{thm: cocycles} and Remark (i) preceding it,
$F = F_{Z, L, W}$
for a skewadjoint operator $Z \in B(\init)$,
an operator $L \in B(\init; \init\ot\noise)$
and an isometry $W \in B(\init\ot\noise)$.
Set
$G_h := V^\ell_{hZ, \sqrt{h} L, W}$ ($h>0$).
Then,
by Proposition~\ref{propn: impl}
and Theorem~\ref{thm: main} (b),
the family of isometries $(G_h)_{h>0}$ in $B(\init\ot\khat)$
satisfies~\eqref{subeq: z} and~\eqref{subeq: a}.
Moreover,
if $X$ is unitary then $W$ is unitary and so each $G_h$ is too, and~\eqref{subeq: b} also holds.

(2)
Suppose next that $X$ is coisometric.
Then
the dual cocycle $\dualco$
is isometric with stochastic generator $F^*$.
Therefore,
by (1),
there is a family of isometries $(\Gtilde_h)_{h>0}$ in $B(\init\ot\khat)$
satisfying
$\hscale ( \Gtilde_h - I ) \to F^*$ as $h \to 0$.
Setting
$G_h = (  \Gtilde_h )^*$ ($h > 0$),
each $G_h$ is coisometric and~\eqref{subeq: z} holds,
so Theorem~\ref{thm: main} (c) implies that~\eqref{subeq: b} holds.

(3)
Suppose finally that $X$ is quasicontractive with exponential growth bound $\beta$.
Then
the contraction cocycle $( e^{- \beta t} X_t )_{t \ges 0}$
has stochastic generator $F - \beta \QIP^\perp$.
Set $\Kil := \noise \op \noise \op \Comp$ and let $J$ be the isometry
$\left[ \begin{smallmatrix} I \\ 0 \\ 0 \end{smallmatrix} \right] \in B(\noise; \Kil)$.
Then, by
Theorem 1.3 of~\cite{LQST2},
there is an
operator $\Ftilde \in B(\init \ot \Kil)$
such that,
in the notation of Lemma~\ref{lemma: JGJ},
\[
\Ftilde^* \seriesprod \Ftilde = 0 =
\Ftilde \seriesprod \Ftilde^*
\ \text{ and } \
F - \beta \QIP^\perp = J_\noise^*  \Ftilde J_\noise.
\]
It follows from part (1) that
there is a family of unitaries
$(\Gtilde_h)_{h>0}$ in $B(\init\ot\ktildehat)$
such that, as $h \to 0$,
\begin{align*}
&
\hscale ( \Gtilde_h - I_{\init \ot \ktildehat} ) \to \Ftilde,
\ \text{ and }
\\
&
\sup_{t\in [0,T] }
\Big(
\big\|
\big(
X^{\Ftilde}_t - X^{h,\Gtilde_h}_t
 \big)\xi
 \big\|
+
\big\|
\big(
X^{\Ftilde}_t - X^{h,\Gtilde_h}_t
 \big)^*
\xi
 \big\|
\Big)
\to 0
\qquad
(\xi \in \init \ot \Fockk, T \in \Rplus).
\end{align*}
Now set
$G_h = e^{\beta h} J_\noise^* \Gtilde_h J_\noise$ ($h > 0$).
Then
\[
\norm{ G_h }^{ \lfloor t/h \rfloor }
\les
 e^{ \beta h \lfloor t/h \rfloor }
\les
 e^{ t \beta }
\qquad
(h > 0, t \in \Rplus),
\]
so~\eqref{subeq: y} holds,
and~\eqref{subeq: z} holds since
\begin{align*}
\hscale (G_h - I )
&
=
\hscale ( e^{\beta h} J_\noise^*  \Gtilde_h  J_\noise - I )
\\
&
=
e^{\beta h}  J_\noise^* \hscale (\Gtilde_h - I ) J_\noise
+
( e^{\beta h} - 1 ) \hscale(I)
\to
 J_\noise^*  \Ftilde J_\noise + \beta \QIP^\perp
= F.
\end{align*}
Moreover~\eqref{subeq: a} and~\eqref{subeq: b} hold since,
by Lemma~\ref{lemma: JGJ},
\begin{align*}
X_t^{F}
-
X_t^{h,G_h}
&
=
X_t^{J_\noise^* ( \Ftilde + \beta \Deltatilde^\perp ) J_\noise}
-
X_t^{h, J_\noise^* e^{\beta h} \Gtilde_h J_\noise}
 \\
&
=
J_\Fock^*
\big(
X_t^{\Ftilde + \beta \Deltatilde^\perp}
-
X_t^{h,e^{\beta h} \Gtilde_h}
\big) J_\Fock
=
J_\Fock^*
\big(
e^{\beta t}
X_t^{\Ftilde}
-
e^{ \beta h \lfloor t/h \rfloor }
X_t^{h,\Gtilde_h}
\big) J_\Fock
\end{align*}
for
$t \in \Rplus$ and $h > 0$,
and $h \lfloor t/h \rfloor \to t$ locally uniformly as $h \to 0$ .
\end{proof}

%This fulfills the promise contained in the remark following Corollary~\ref{cor: inner}.

%%%%%%%%%%%%%%%%%%%%%%%%%%%%%%%%%%%      EXAMPLE

\section{Repeated quantum interactions}
\label{section: repeated quantum interactions}

In this section we consider repeated quantum interactions and the entanglement of bipartite systems,
and prove two theorems
demonstrating how these fit into the theory developed in the preceding sections.
Whereas in the previous section we started with a quantum stochastic evolution
and showed how to realise it
as a limit of quantum random walks,
in this section we travel in the opposite direction and,
starting with a discrete quantum dynamics we first show how,
through the appropriate scaling,
one obtains a limiting continuous-time dynamics.
We then treat the case where the discrete dynamics is given by
a composition consisting of two systems
which are physically independent
and are
 separately interacting with a common environment.

In the model developed by Attal and Pautrat (\cite{AtP}),
one has a family of discrete-time evolutions of an open quantum system
consisting of a system $\Sil$ with its Hamiltonian $H_\Sil$,
coupled to an environment modeled by
an infinite chain of identical particles
with each particle governed by a Hamiltonian $H_\Pil$,
repeatedly interacting with the system $\Sil$ over a short time period of length $h$,
through an interaction Hamiltonian $H_\Il(h)$.
Specifically, one takes
%\begin{example}
%\label{Example1}
%In the model developed by Attal and Pautrat (\cite{AtP})
\[
G_h
=
e^{-i h H_\Til(h)}
\qquad
(h > 0),
\]
where the \emph{total Hamiltonian}
decomposes as
\[
H_\Til(h) =
H_\Sil \ot I_\khat + I_{\init} \ot H_\Pil + H_\Il(h)
\]
for a \emph{system Hamiltonian}
 $H_\Sil \in B(\init)_{\sa}$,
a \emph{particle Hamiltonian}
 $H_\Pil \in B(\khat)_{\sa}$ and
 an \emph{interaction Hamiltonian}
taking the form
\[
H_\Il(h) =
\frac{1}{h}
\begin{bmatrix}
0 &  \sqrt{h} V_\Diil^* \\ \sqrt{h} V_{\Diil} &  H_{\Scil}
\end{bmatrix}
\]
for operators $V_\Diil \in B(\init; \init \ot \noise)$ and $H_\Scil \in B(\init \ot \noise)_{\sa}$.

This fits perfectly into the general scheme described here.
Indeed,
\[
 - i h H_\Til (h)
=
-i \,
\begin{bmatrix} h H_\Sil &  \sqrt{h} V_\Diil^* \\ \sqrt{h} V_{\Diil} &  H_{\Scil} \end{bmatrix}
- ih  I_\init \ot H_\Pil   - i h ( 0_\init \op (H_\Sil \ot I_\noise) )
\]
so,
under the scaling~\eqref{eqn: scaling},
\[
\hscale(  - i h H_\Til (h) ) \to
- i \,
\begin{bmatrix}  H_\Sil + \omega ( H_\Pil ) I_\init  &  V_\Diil^* \\ V_{\Diil} &  H_{\Scil} \end{bmatrix}
\ \text{ as } \ h \to 0,
\]
where $\omega := \omega_{\wh{0}}$, the vector state corresponding to the vector $\wh{0} \in \khat$.
Therefore, by
Theorem~\ref{thm: Ghn}
and
Proposition~\ref{propn: Q skew},
we have the following
strong convergence of a scaled unitary QRW to a QS unitary cocycle.

\begin{thm}
\label{thm: rqi}
Let
$G_h = e^{- i h H_\Til (h) }$
where
$ H_\Til (h)$ is as above, for $h>0$.
Then
\[
\sup_{t\in [0,T]}
\big\|
\big(
X^{h,G_h }_t - X^F_t
 \big)\xi
 \big\|
\to 0
\text{ as }
h \to 0
\qquad
(\xi \in \init \ot \Fockk, T \in \Rplus)
\]
and,
in the notations~\eqref{eqn: Flft} and~\eqref{eqn: entire},
$F = F_{-iH, L, W}$
for the operators
\[
H = H_\Sil + \omega ( H_\Pil ) I_\init - iV_\Diil^*\,  e ( - i H_{\Scil} ) V_\Diil,
\
\
L = -i e_1(-i H_\Scil) V_{\Diil},
\
\text{ and }
\
W = e_0 ( -i H_\Scil ).
\]
\end{thm}

\begin{rems}
This result implies Theorem 19 of~\cite{AtP} in coordinate-free form,
with the difference that
here no Hilbert--Schmidt-type conditions need be imposed on the
matrix components of $L$ and $W$ with respect to some fixed orthonormal basis of
the noise dimension space $\noise$.

For a discussion of the physical origins of the components of the interaction Hamiltonian
see~\cite{Quant}.
In brief,
 the scaling order $\sqrt{h}$ corresponds to a \emph{weak coupling limit},
or van Hove limit (\cite{vanHov},~\cite{Davies}),
whereas the scaling order $h$
corresponds to a \emph{low density limit} (\cite{Dum}).

In case the interaction Hamiltonian has no scattering component,
$F$ takes the following simpler form
\[
\begin{bmatrix}
- i ( H_\Sil + \omega ( H_\Pil ) I_\init ) - \tfrac{1}{2}( V_\Diil )^* V_\Diil
&  -i( V_\Diil )^*
\\  -i V_\Diil
& 0
\end{bmatrix}.
\]
On the other hand,
in case there is no dipole term in the interaction Hamiltonian,
so that it is purely scattering,
the operators $H_\Til(h)$ and $F$ then take the respective forms
\[
\begin{bmatrix}
 H_\Sil & 0 \\ 0 & h^{-1} H_\Scil +  ( H_\Sil \ot I_\noise )
\end{bmatrix}
+
 I_\init \ot H_\Pil
\ \text{ and } \
\begin{bmatrix}
- i ( H_\Sil + \omega ( H_\Pil ) I_\init ) & 0 \\ 0 & e^{-iH_\Scil} - I_{\init \ot \noise}
\end{bmatrix}.
\]
Thus, if also $H_\Sil = 0$,
then $X^F = ( e^{-it \omega (H_\Pil ) } U_t )_{t \in \Rplus}$
where $U$ is a unitary QS cocycle of preservation type, as described in Example~\ref{eg: pres}.
\end{rems}
%\end{example}

%\begin{example}
%\label{example: 8.6}
We now turn to
the model of entanglement of bipartite systems under
repeated quantum interactions studied by Attal, Deschampes and Pelligrini (\cite{ADP}).
Here
the system space $\init$ is a tensor product $\init_1 \ot \init_2$ of constituent system spaces,
and $G_h = G_1(h) G_2(h)$ where
	\begin{align*}
		G_2(h) =& e^{-i h I_{1} \ \ot \  H^{(2)}_\Til(h)} = I_{1} \ot 	e^{-i h H^{(2)}_\Til(h)}
\\
G_1(h) =& e^{-i h I_{2} \ \widetilde{\ot} \ H^{(1)}_\Til(h)} = I_{2}  \ \widetilde{\ot} \	e^{-i h H^{(1)}_\Til(h)},
			\end{align*}
in which
$ I_{1}:= I_{\init_1}$, $I_2 := I_{\init_2}$,
with the tilde capturing the tensor flip from $B(\init_2) \otol B(\init_1 \ot \noise)$ to $B(\init \ot \noise)$
(as in Example~\ref{ADPexample}),
and the
total Hamiltonians
	decompose as
	\begin{equation*}
		H^{(i)}_\Til(h) =
		H^{(i)}_\Sil \ot I_{\khat} + I_{1} \ot H_\Pil + H^{(i)}_\Il(h)
\qquad
(i=1,2)
	\end{equation*}
	for system Hamiltonians
	$H^{(i)}_\Sil \in B(\init_i)_{\sa}$ ($i=1,2$),
	a single particle Hamiltonian
	$H_\Pil \in B(\khat)_{\sa}$ and
	interaction Hamiltonians
	taking the form
	\[
	H^{(i)}_\Il(h) =
	\frac{1}{h}
	\begin{bmatrix}
	0 &  \sqrt{h}   \left( V^{(i)}_{\Diil} \right)^* \\ \sqrt{h}  V^{(i)}_{\Diil} &  H^{(i)}_\Scil
	\end{bmatrix}
	\]
	for operators $ V^{(i)}_{\Diil} \in B(\init_i; \init_i \ot \noise)$
and
$H^{(i)}_\Scil \in B(\init_i \ot \noise)_{\sa}$
($i=1,2$).
	From the preceding example we deduce that
(again setting $\omega := \omega_{\wh{0}}$),
	\begin{equation*}
		\hscale(  - i h H^{(i)}_\Til (h) ) \to
		- i \,
		\begin{bmatrix}  H^{(i)}_\Sil + \omega ( H_\Pil ) I_{i}  &  \left( V^{(i)}_\Diil \right)^* \\ V^{(i)}_{\Diil} & H^{(i)}_\Scil \end{bmatrix}				
\ \text{ as } \ h \to 0
\qquad
(i=1,2).
	\end{equation*}

\begin{thm}
\label{thm: bipart}
Let
$G_1(h)$ and $G_2(h)$ be as above,
for $h>0$.
Then
	\[
	\sup_{t\in [0,T]}
	\big\|
	\big(
	X^{h, G_1(h)G_2(h) }_t - X^{F_1 \seriesprod F_2}_t
	\big)\xi
	\big\|
	\to 0
	\text{ as }
	h \to 0
	\qquad
	(\xi \in \init \ot \Fockk, T \in \Rplus)
	\]
where $F_2 := I_{1} \ot F_{(2)}$,
	$F_1 := I_{2} \flip F_{(1)}$, and
	in the notations~\eqref{eqn: entire} and~\eqref{eqn: Flft},
	$F_{(i)} = F_{-iH^{(i)}, L^{(i)}, W^{(i)}}$
	for the operators
	\begin{align*}
&
H^{(i)} = H^{(i)}_\Sil + \omega ( H_\Pil ) I_{\init_i}  - i( V_\Diil^{(i)} )^*  e ( - i H_{\Scil}^{(i)} ) V_\Diil^{(i)},
\\
&
L^{(i)} = -i e_1(-i H_{\Scil}^{(i)}) V_\Diil^{(i)},
		\
		 \text{ and } \
W^{(i)}=\ e_0 ( -i H_\Scil^{(i)} ).
	\end{align*}
Moreover,
$F_1 \seriesprod F_2 = F_{-iH, L, W}$
for the operators
\begin{align*}
&
H = H^{(1)} \ot I_2 + I_1 \ot H^{(2)} +
\im
\Big(
\left[
I_2 \, \wt{\ot} \, \big( V^{(1)}_{\Diil} \big)^* e_1 ( -i  H^{(1)}_\Scil )
\right]
\left[ I_1 \ot e_1 ( -i H^{(2)}_\Scil ) V^{(2)}_{\Diil} \right]
\Big),
\\
&
L = -i \Big(
I_2 \, \wt{\ot} \, e_1 ( -i  H^{(1)}_\Scil ) V^{(1)}_{\Diil}
+
\big[ I_2 \, \wt{\ot} \, e_0 ( -i  H^{(1)}_\Scil ) \big]
\big[ I_1 \ot e_1 ( -i H^{(2)}_\Scil ) V^{(2)}_{\Diil} \big]
\Big),
\
\text{ and }
\\
&
W = e_0 \left( -i I_2 \, \wt{\ot} \, H^{(1)}_\Scil \right) e_0 \left( -i I_1 \ot H^{(2)}_\Scil \right).
\end{align*}
\end{thm}
\begin{proof}
The first part follows from
	Theorem~\ref{thm: Ghn}
and
	Propositions~\ref{propn: Q skew} and~\ref{propn: composition}.
The second part follows from
identity~\eqref{eqn: Fldecomp} and the relation $e_1(-z)e_0(z) = e_1(z)$ ($z\in\Comp$).	
\end{proof}

\begin{rems}
In view of Proposition~\ref{propn: M and M},
the limiting cocycle $X^{F_1 \seriesprod F_2}$
is actually the (pointwise) product of the individual cocycles
$X^{F_1}$ and $X^{F_2}$
where $F_1, F_2 \in B(\init \ot \khat)$ are as above.

If we assume that neither of the interaction Hamiltonians has a scattering component:
 $H^{(1)}_\Scil = H^{(2)}_\Scil = 0$,
 then $F$ takes the form
 $\left[\begin{smallmatrix}
 -iH - \tfrac{1}{2} L^*L & -L^* \\ L & 0
 \end{smallmatrix}\right]$
 where
\begin{align*}
&
H =
 H^{(1)}_\Sil \ot I_2 + I_1 \ot H^{(2)}_\Sil + 2 \omega ( H_\Pil ) I_1  \ot I_2 +
\im
\big(
\big[
I_2 \, \wt{\ot} \, \big( V^{(1)}_{\Diil} \big)^*
\big]
\big[ I_1 \ot  V^{(2)}_{\Diil} \big]
\big)
\ \text{ and } \
\\
&
L = -i \big(
I_2 \, \wt{\ot}  V^{(1)}_{\Diil}
+
 I_1 \ot  V^{(2)}_{\Diil}
\big).
 \end{align*}
 Assuming further that
 the noise dimension space
 $\noise$ is finite dimensional,
 with fixed orthonormal basis $(e_j)_{1 \les j \les d}$,
 and
 setting
 $I_A := I_{\init_1}$,
 $I_B := I_{\init_2}$,
 $H^A := H^{(1)}$,
 $H^B := H^{(2)}$,
 \[
 \lambda_0 := \omega(H_\Pil),
 \
 V_j := \left( I_{\init_1 } \ot \left<e_j \right| \right) V^{(1)}_{\Diil}
 \
 \text{ and }
 W_j :=\left( I_{\init_2 } \ot \left< e_j \right|\right) V^{(2)}_{\Diil},
\ \text{ for } j = 1, \cdots , d,
\]
one gets
\[
F_1 \seriesprod F_2
=
\begin{bmatrix}
K
&
L_1^*
&
\cdots
&
L_d^*
\\
L_1
& 0 & \cdots & 0
\\
\vdots
& \vdots & \ddots & \vdots
%\\
%.
%& . & & .
\\
L_d
& 0 & \cdots & 0
\end{bmatrix}
\in B(\init \ot \Comp^{d+1}),
\]
where
\begin{align*}
&
L_j := -i ( V_j \ot I_B + I_A \ot W_j )
\ \text{ for } j = 1, \cdots , d,
\text{ and }
\\
&
K :=
-i \big( H^A \ot I_B + I_A \ot H^B + 2 \lambda_0 \, I_A \ot I_B \big)
\\
&
\qquad \qquad \qquad \qquad \qquad
-\frac{1}{2}
\sum_{j=1}^d
\big(
V_j^* V_j \ot I_B + I_A \ot W_j^* W_j
\big)
+
\sum_{j=1}^d V_j^* \ot W_j
\end{align*}
and so,
 modulo the fact that we work with left (rather than right) cocycles,
 we recover
 Theorem 3.1 of~\cite{ADP} as a special case of Theorem~\ref{thm: bipart}.
\end{rems}

%\end{example}

%%%%%%%%%%%%%%%%%%%%%%%%%%%%%%%%%%%%%%%%%%%% SECTION

%%%%%%%%%%%%%%%%%%%%%%%%%%%%%%%%%%%      ACKNOWLEDGEMENTS
  \par\medskip\noindent
  {\bf Acknowledgements.}
  We are grateful to anonymous referees for
  insightful remarks and a very thorough engagement with the paper.
  This work benefited from the support of the
  UK-India Education and Research Initiative grant QP-NCG-QI:
  \emph{Quantum Probability,
 Noncommutative Geometry and
 Quantum  Information}.

%%%%%%%%%%%%%%%%%%%%%%%%%%%%%%%%%%%      BIBLIOGRAPHY


\begin{thebibliography}{$\text{BJM}$}

\bibitem[AcB]
{AcB89}
L.~Accardi and A.~Bach,
Central limits of squeezing operators,
\emph{in}
``Quantum Probability and Applications IV,''
(\emph{Eds}.\ L.~Accardi \& W.~von Waldenfels),
\emph{Lecture Notes in Math.}\ \textbf{1396},
Springer, Berlin, 1989, 7--19.


\bibitem[AFL]
{AFL}
L.~Accardi, A.~Frigerio and Y.G.~Lu,
The weak coupling limit as a quantum functional central limit,
\emph{Comm.\ Math.\ Phys.}\
\textbf{131} (1990), no.~3, 537--570.


\bibitem[AA+]
{AAKV}
D.~Aharonov, A.~Ambainis, J.~Kempe, U.~Vazirani,
Quantum walks on graphs,
\emph{in}
``Proceedings, Thirty-Third Annual ACM Symposium on Theory of Computing,''
ACM, New York, 2001, 50--59.

\bibitem[ADP]
{ADP}
S.~Attal, J.~Deschamps and C.~Pellegrini,
Entanglement of bipartite quantum systems driven by repeated interactions,
\emph{J.\ Statist.\ Phys.}\
\textbf{154} (2014), no.~3, 819--837.

\bibitem[$\text{AtJ}$]
{AJrepeated}
S.~Attal and A.~Joye,
Weak coupling and continuous limits for repeated quantum interactions,
\emph{J.\ Stat.\ Phys.}\
\textbf{126} (2007), no.~6, 1241--1283.

\bibitem[AtP]
{AtP}
S.~Attal and Y.~Pautrat,
From repeated to continuous quantum interactions,
\emph{Ann.\ Henri Poincar\'e}
\textbf{7} (2006), no.~1, 59--104.


\bibitem[At+]
{At+}
S.~Attal, F.~Petruccione, C.~Sabot and I.~Sinayskiy,
Open quantum random walks,
\emph{J. Stat. Phys.}\
\textbf{147} (2012), no.~4, 832–-852.

\bibitem [Bel] {Belavkin}
 V.P. Belavkin,
 A new form and a *-algebraic structure of
 quantum stochastic integrals in Fock space,
 \emph{Rend. Sem. Mat. Fis. Milano}
 \textbf{58} (1988), 177--193.


\bibitem[$\text{B}_{1}$]
{BQRW}
A.C.R.~Belton,
Random-walk approximation to vacuum cocycles,
\emph{J.\ Lond.\ Math.\ Soc.}~(2)
\textbf{81} (2010), no.~2,
412--434.

\bibitem[$\text{B}_{2}$]
{B thermalisation}
  --- --- ,
Quantum random walks and thermalisation,
\emph{Comm.\ Math.\ Phys.}\
\textbf{300} (2010), no.~2, 317--329.

\bibitem[$\text{B}_{3}$]
{B thermalisation 2}
  --- --- ,
Quantum random walks with general particle states,
\emph{Comm.\ Math.\ Phys.}
\textbf{328} (2014), no.~2, 573--596.

\bibitem[$\text{BG+}$]
{BGLZ}
A.C.R.~Belton, M.~Gnacik, J.M.~Lindsay and Ping Zhong,
Quasifree stochastic cocycles and quantum random walks,
\emph{arXiv}:1704.00682 [math-ph].

\bibitem[BvH]{BvH08}
L.~Bouten and R.~van Handel,
Discrete approximation of quantum stochastic models,
\emph{J.\ Math.\ Phys.}\
\textbf{49} (2008) 102109, 19 pp.

\bibitem[Bv+]{BvHJ09}
L.~Bouten, R.~van Handel and M.R.~James,
A discrete invitation to quantum filtering and feedback control,
\emph{SIAM Review}
\textbf{51} (2009), no.~2, 239--316.

\bibitem[BPe]{Quant}
H. P.~Breuer and F.~Petruccione,
``The Theory of Open Quantum Systems,''
Oxford University Press, Oxford, 2002.

\bibitem[BJM]{BJM06}
L.~Bruneau, A.~Joye and M.~Merkli,
Asymptotics of repeated interaction quantum systems,
\emph{J.\ Funct.\ Anal.}\
\textbf{239} (2006), no.~1, 310--344.

\bibitem[BPi]{BrP09}
L.~Bruneau and C.-A.~Pillet, Thermal relaxation of a QED cavity,
\emph{J.~Stat.\ Phys.}\
\textbf{134} (2009), no.~5--6, 1071--1095.

\bibitem[Chi]{Childs}
A.~Childs,
On the relationship between continuous- and discrete-time quantum walk,
\emph{Comm.\ Math.\ Phys.}
\textbf{294} (2010), no.~2, 581--603.

\bibitem[$\text{DL}_{1}$]
{DL1}
B.K.~Das and J.M.~ Lindsay,
Elementary evolutions in Banach algebra,
\emph{J.\ Ramanujan Math.\ Soc.}\
\textbf{28} (2013), no.~4, 483--497.

\bibitem[$\text{DL}_{2}$]
{DaL}
  --- --- ,
Quantum random walk approximation in Banach algebra,
\emph{J.\ Math.\ Anal.\ Appl.}\
\textbf{430} (2015), no.~1, 465--482.

\bibitem[Dav]{Davies}
E.B.~Davies,
Markovian master equations,
\emph{Comm.\ Math.\ Phys.}\
\textbf{39} (1974), no.~2, 91--110.

\bibitem[D\"um]{Dum}
R.~D\"umcke,
The low density limit for an N-level system interacting with a free Bose or Fermi gas,
\emph{Comm.\ Math.\ Phys.}\
\textbf{97} (1985), no.~3, 331--359.

\bibitem[EfR]
{EfR}
E.G.~Effros and Z.-J.~Ruan,
``Operator Spaces,''
 Oxford University Press,
 Oxford, 2000.

 \bibitem
 [$\text{Fa}_{1}$]
{Fagnola}
 F. Fagnola,
 Characterization of isometric and unitary weakly differentiable
 cocycles in Fock space,
 \emph{in},
 ``Quantum Probability \& Related Topics,'' QP-PQ VIII,
 (Ed. L. Accardi),
 World Scientific, Singapore, 1993,
 pp. 143--164.

\bibitem
[$\text{Fa}_{2}$]
{FagnolaNotes}
--- --- ,
Quantum Markov semigroups and quantum flows,
 \emph{Proyecciones}
 \textbf{18} (1999) no.\ 3, 1--144.

\bibitem[FrS]
{FrS}
U.~Franz and A.~Skalski,
Approximation of quantum L\'evy processes by quantum random walks,
\emph{Proc.\ Indian Acad.\ Sci.\ Math.\ Sci.}\
\textbf{118} (2008), no.~2, 281--288.

  \bibitem [GaZ]
 {Gardiner}
  C.W.~Gardiner and P.~Zoller,
  ``Quantum Noise.
  A Handbook of Markovian and Non-Markovian Quantum Stochastic Methods
  with Applications to Quantum Optics''
  (3rd edn.),
  \emph{Springer Series in Synergetics} \textbf{1538},
  Springer, Berlin, 2004.

\bibitem[Goh]{Goh10}
R.~Gohm,
Non-commutative Markov chains and multi-analytic operators,
\emph{J.\ Math.\ Anal.\ Appl.}\
\textbf{364} (2010), no.~1, 275--288.

\bibitem[Gou]{Gough}
J.~Gough,
Holevo-ordering and the continuous-time limit for open Floquet
dynamics,
\emph{Lett.\ Math.\ Phys.}\
\textbf{67} (2004), no.~3, 207--221.

\bibitem[GoS]{GoS04}
J.~Gough and A.~Sobolev,
Stochastic Schr\"odinger equations as limit of discrete filtering,
\emph{Open Syst.\ Inf.\ Dyn.}
\textbf{11} (2004), no.~3, 235--255.

\bibitem[$\text{Ho}_{1}$]
{Holevo1}
A.S.~Holevo,
Time-ordered exponentials in quantum stochastic calculus,
\emph{in}
``Quantum Probability and Applications VII,''
(\emph{Ed.}\ L.~Accardi),
World Scientific, Singapore, 1992, 175--202.

\bibitem[$\text{Ho}_{2}$]
{Holevo2}
--- --- ,
Exponential formulae in quantum stochastic calcuclus,
\emph{Proc.\ Roy.\ Soc.\ Edinburgh, Sect.\ A}
\textbf{126} (1996), no. 2, 375–-389.

\bibitem[$\text{Ho}_{3}$]
{Holevo3}
--- --- ,
``Statistical Structure of Quantum Theory,''
\emph{Lecture Notes in Physics}, Monographs, \textbf{67},
Springer-Verlag, Berlin, 2001.

\bibitem[$\text{Ho}_{4}$]
{Holevo4}
--- --- ,
Quantum and classical stochastic calculus,
\emph{in}
``Quantum Probability Communications XI,''
(\emph{Eds.}\ S.~Attal and J.M.~Lindsay),
World Scientific, Singapore, 1992, 199--299.


\bibitem[Jou]
{Jou}
J.-L.~Journ\'e,
Structure des cocycles markoviens sur l'espace de Fock,
\emph{Probab.\ Theory Rel.\ Fields}
\textbf{75} (1987) no.\ 2, 291--316.

\bibitem[JuL]
{JuL}
M.~Jurczy\'nski and J.M.~Lindsay,
Quantum Wiener chaos, \emph{in preparation}.

\bibitem[Kem]
{Kem}
J.~Kempe,
Quantum random walks --- an introductory overview,
\emph{Contemporary Physics}
\textbf{44} (2003), no.~4, 307--327.

\bibitem [Kon]
{Konno}
N.~Konno,
Quantum walks,
\emph{in}
``Quantum Potential Theory,''
 (\emph{Eds}.\ U.~Franz \& M.~Sch\"{u}rmann  ),
\emph{Lecture Notes in Math.}\ \textbf{1954},
Springer, Berlin, 2008, 309--452.

\bibitem [$\text{L}_{1}$]
{Lbook}
J.M.~Lindsay,
Quantum stochastic analysis --- an introduction,
\emph{in}
``Quantum Independent Increment Processes I,''
 (\emph{Eds}.\ M.~Sch\"{u}rmann \& U.~Franz),
\emph{Lecture Notes in Math.}\ \textbf{1865},
Springer, Berlin, 2005, 181--271.

\bibitem [$\text{L}_{2}$]
{LQST2}
--- --- ,
Quantum stochastic Lie--Trotter product formula II,
 \emph{Int.\ Math.\ Res.\ Not.\ IMRN} (to appear),
  \emph{arXiv}:1707.05669v3 [math.FA].

\bibitem [LiM]
{LM2}
J.M.~Lindsay and O.T.~Margetts,
Quasifree stochastic analysis,
\emph{Preprint}.

\bibitem[LiP]
{LiP}
J.M.~Lindsay and K.R.~Parthasarathy,
The passage from random walk to diffusion in quantum probability II,
\emph{Sankhy\={a} Ser.~A}
\textbf{50} (1988), no.~2, 151--170.

\bibitem[LiS]
{LiS}
J.M.~Lindsay and A.G.~Skalski,
Quantum random walk approximation on locally compact quantum groups,
\emph{Lett.\ Math.\ Phys.}\
\textbf{103} (2013), no.~7, 765--775.

\bibitem [$\text{LW}\!_{1}$]
{LW1}
 J.M. Lindsay and S.J. Wills,
 Existence, positivity, and contractivity for
 quantum stochastic flows with infinite dimensional noise,
 \emph{Probab. Theory Rel. Fields}
 \textbf{116} (2000) no. 4, 505--543.

\bibitem [$\text{LW}\!_{2}$]
{LW2}
  --- --- ,
Markovian cocycles on operator algebras, adapted to a Fock filtration,
\emph{J. Funct.\ Anal.}\
\textbf{178} (2000), no.\ 2, 269--305.

 \bibitem [$\text{LW}\!_3$]
{LiW}
  --- --- ,
Quantum stochastic operator cocycles via associated semigroups,
\emph{Math.\ Proc. Camb.\ Phil.\ Soc.}\
 \textbf{142} (2007), no.\ 3, 535--556.

\bibitem[$\text{Me}_{1}$]
{Mey89}
P.-A.~Meyer,
\'El\'ements de probabilit\'es quantiques. X.
Approximation de l'oscillateur harmonique (d'apr\`es L.~Accardi et
A.~Bach),
\emph{in}
``S\'eminaire de Probabilit\'es XXIII,''
(\emph{Eds}.\ J.~Az\'ema, P.-A.~Meyer \& M.~Yor),
\emph{Lecture Notes in Math.}\ \textbf{1372},
Springer, Berlin, 1989, 175--182.

\bibitem [$\text{Me}_{2}$]
{Meyer}
--- --- ,
``Quantum Probability for Probabilists''
(2nd edn.),
\emph{Lecture Notes in Math.} \textbf{1538},
Springer, Berlin, 1995.

\bibitem
[$\text{Pa}_{1}$]
{Par}
K.R.~Parthasarathy,
The passage from random walk to diffusion in quantum probability,
\emph{J.\ Appl.\ Probab.} \textbf{25A} (1988), no.~2, 151--166.

\bibitem
[$\text{Pa}_{2}$]
{Partha}
--- --- ,
``An Introduction to Quantum Stochastic Calculus'',
\emph{Monographs in Math}., Birkh{\"a}user Verlag, Basil, 1991.

\bibitem[Pel]
{Pell}
C.~Pellegrini,
Continuous time open quantum random walks and non-Markovian master equations,
\emph{J.~Stat.\ Phys.}\
\textbf{154} (2014), no.~3, 838--865.

 \bibitem [ReS]
 {ReS}
 M.~Reed and B.~Simon,
 ``Methods of Modern Mathematical Physics, I: Functional Analysis
 (2nd Edn.), II: Fourier Analysis, Self-Adjointness,''
 Academic Press, New York, 1980, 1975.

 \bibitem [RSz] {RiN}
 F.~Riesz and B.~Sz.-Nagy,
 ``Functional Analysis''
 (transl.~from 2nd French Edn.),
 \emph{Dover Books on Advanced Mathematics},
 Dover Publications, New York, 1990.

\bibitem[Sah]
{Sah}
L.~Sahu,
Quantum random walks and their convergence to Evans-Hudson flows,
\emph{Proc.\ Indian Acad.\ Sci.\ Math.\ Sci.}\
\textbf{118} (2008), no.~3, 443--465.

\bibitem[Sin]
{Sinha}
K.B.~Sinha,
Quantum random walk revisited,
\emph{in}
``Quantum Probability,''
\emph{Banach Centre Publications}
\textbf{76} (2006), 377--390.

\bibitem[Ske]
{Skeide}
 M.~Skeide,
Indicator functions of intervals are totalising in the symmetric Fock
space $\Gamma( L^2(\Rplus) )$,
\emph{in}
``Trends in Contemporary Infinite Dimensional Analysis and Quantum Probability.
Volume in Honour of Takeyuki Hida,''
(\emph{Eds}.\ L.~Accardi, H.-H.~Kuo, N.~Obata, K.~Saito, Si~Si and L.~Streit),
Istituto Italiano di Cultura, Kyoto, 2000.

\bibitem[Sko]
{Skorohod}
A.V.~Skorohod,
``Asymptotic Methods in the Theory of Stochastic Differential Equations,''
\emph{Transl. Mathematical Monographs},
\textbf{78}
American Mathematical Society, Providence, RI, 1989.

\bibitem[vHo]{vanHov}
L.~van~Hove,
Quantum-mechanical perturbations giving rise to a statistical
transport equation, \emph{Physica}
\textbf{21} (1955), no.~1--5, 517--540.

\bibitem[vWa]
{vWa84}
W.~von Waldenfels,
It\^o solution of the linear quantum stochastic differential equation
describing light emmission and absorbtion,
\emph{in}
``Quantum Probability and Applications to the Quantum Theory of Irreversible Processes,''
(\emph{Eds}.\ L.~Accardi, A.~Frigerio \& V.~Gorini),
\emph{Lecture Notes in Math.}\ \textbf{1055},
Springer, Berlin, 1984, 384--411.

\bibitem[VDN]
{VDN}
D.~Voiculescu, K.J.~Dykema \& A.~Nica,
``Free random variables,
A noncommutative probability approach to free products
with applications to
random matrices, operator algebras and harmonic analysis on free groups,''
\emph{CRM Monograph Series},
American Mathematical Society, Providence, RI,
1992.
  %  $ deliberate mistake

 \end{thebibliography}
 \end{document}